\documentclass{siamltex}
\usepackage{amssymb}
\usepackage{amscd}
\usepackage{bm}
\usepackage{amsmath}
\usepackage{graphicx}
\usepackage{subfigure}
\usepackage{float}

\newtheorem{lem}{Lemma}[section]
\newtheorem{thm}{Theorem}[section]
\newtheorem{alg}{Algorithm}[section]

\newtheorem{exmp}{Example}[section]

\renewcommand{\Re}{\mathrm{Re}\,}
\renewcommand{\Im}{\mathrm{Im}\,}

\newcommand{\pa}{\partial}

\newcommand{\na}{\nabla}

\newcommand{\ga}{\gamma}
\newcommand{\Ga}{\Gamma}
\newcommand{\Om}{\Omega}
\newcommand{\de}{\delta}

\newcommand{\De}{\Delta}

\newcommand{\lam}{\lambda}

\newcommand{\bks}{\backslash}

\newcommand{\pl}{{\rm pl}}
\newcommand{\ps}{{\rm ps}}

\renewcommand{\i}{\mathbf{i}}
\newcommand{\R}{{\mathbb{R}}}

\renewcommand{\Re}{\mathrm{Re}\,}
\renewcommand{\Im}{\mathrm{Im}\,}

\def\eps{{\varepsilon}}

\renewcommand{\i}{\mathbf{i}}
\newcommand{\be}{\begin{eqnarray}}
\newcommand{\ee}{\end{eqnarray}}
\newcommand{\ben}{\begin{eqnarray*}}
\newcommand{\een}{\end{eqnarray*}}
\newcommand{\bee}{\begin{equation}}
\newcommand{\eee}{\end{equation}}

\newcommand{\debproof}{\begin{proof}}
\newcommand{\finproof}{\end{proof}}

\title{Phaseless Imaging by Reverse Time Migration: Acoustic Waves
\thanks{This work is  supported by National Basic Research Project under the grant 2011CB309700 and China NSF under the grants 11021101 and 11321061.}}

\author{Zhiming Chen\thanks{LSEC, Institute of Computational Mathematics and Scientific Engineering Computing,
Academy of Mathematics and System Sciences, Chinese Academy of Sciences, Beijing 100190, P.R. CHINA ({\tt zmchen@lsec.cc.ac.cn}).}
        \and Guanghui Huang\thanks{The Rice Inversion Project, Department of Computational and Applied
Mathematics, Rice University, Houston, TX 77005-1892, USA ({\tt ghhuang@rice.edu}). Previously: LSEC, Institute of Computational Mathematics and Scientific Engineering Computing,
Academy of Mathematics and System Sciences, Chinese Academy of Sciences, Beijing 100190, P.R. CHINA ({\tt ghhuang@lsec.cc.ac.cn}).}}

\begin{document}
\maketitle

\begin{abstract}
We propose a reliable direct imaging method based on the reverse time migration for finding extended obstacles with phaseless total field data. We prove that the imaging resolution of the method is essentially the same as the imaging results using the scattering data with full phase information. The imaginary part of the cross-correlation imaging functional always peaks on the boundary of the obstacle. Numerical experiments are included to illustrate the powerful imaging quality.
\end{abstract}

\section{Introduction}

We consider in this paper inverse scattering problems with phaseless data which aim to find the support of unknown obstacles embedded in a known background medium from the knowledge of the amplitude of the total field measured on a given surface far away from the obstacles. Let the sound soft obstacle occupy a bounded Lipschitz domain $D\subset\R^2$ with $\nu$ the unit outer normal to its boundary $\Ga_D$.
Let $u^i$ be the incident wave and the total field is $u=u^i+u^s$ with $u^s$ being the solution of the following acoustic scattering problem:
\be
& &\De u^s+k^2u^s= 0\ \ \ \ \mbox{in }\R^2\bks\bar{D},\label{p1}\\
&  & u^s=-u^i\ \ \ \ \ \mbox{on }\Ga_D ,\label{p2} \\
& &  \sqrt{r}\Big(\frac{\pa u^s}{\pa r}-\i k u^s\Big)\rightarrow 0  \ \ \mbox{as } \ \ r=|x|\rightarrow+\infty, \label{p3}
\ee
where $k>0$ is the wave number. The condition \eqref{p3} is the outgoing Sommerfeld radiation condition which guarentees the uniqueness of the solution. In this paper, by the radiation or scattering solution we always mean the solution satisfies the Sommerfeld radiation condition \eqref{p3}. For the sake of the simplicity, we mainly consider the imaging of
sound soft obstacles in this paper. Our algorithm does not require any a priori information of the physical properties of the obstacles such as penetrable or non-penetrable, and for non-penetrable obstacles,
the type of boundary conditions on the boundary of the obstacle. The extension of our theoretical results for imaging other types of obstacles will be briefly considered in section 4.

In the diffractive optics imaging and radar imaging systems, it is much easier to measure the intensity of the total field than the phase information of the field \cite{d89, dbc, zll}. It is thus very desirable to develop reliable numerical methods for reconstructing obstacles with only phaseless data, that is, the amplitude information of the total field $|u|$. In recent years, there have been considerable efforts in the literature to solve the inverse scattering problems with phaseless data. One approach is to image the object with the phaseless data directly in the inversion algorithm, see e.g. \cite{dbc, LB}. The other approach is first to apply the phase retrieval algorithm to extract the phase information of the scattering field from the measurement of the intensity and then use the retrieved full field data in the classical imaging algorithms, see e.g. \cite{fda}. We also refer to \cite{bll} for the continuation method and \cite{kr, LK2010, LK2011} for inverse scattering
  problems with the data of the amplitude of the far field pattern. In \cite{k14} some uniqueness results for phaseless inverse scattering problems have been obtained.

The reverse time migration (RTM)  method, which consists of back-propagating the complex conjugated scattering field into the background medium and computing the cross-correlation between the incident wave field and the backpropagated field to output the final imaging profile, is nowadays a standard imaging technique widely used in seismic imaging \cite{bcs}. In \cite{cch_a, cch_e, ch}, the RTM method for reconstructing extended targets using acoustic, electromagnetic and elastic waves at a fixed frequency is proposed and studied. The resolution analysis in \cite{cch_a, cch_e, ch} is achieved without using the small inclusion or geometrical optics assumption previously made in the literature.

In this paper we propose a direct imaging algorithm based on reverse time migration for imaging obstacles
with only intensity measurement $|u|$ with point source excitations. We prove that the imaging resolution of the new algorithm is essentially the same as the imaging results using the scattering data with the full phase information, that is, our imaging function always peaks on the boundary of the obstacles. To the
best knowledge of the authors, our method seems to be the first attempt in applying non-iterative method for reconstructing obstacles with phaseless data except \cite{d89} in which a direct method is considered for imaging a penetrable obstacle under Born approximation using plane wave incidences. We will extend the RTM method studied in this paper for electromagnetic probe waves in a future paper.

The rest of this paper is organized as follows. In section 2 we introduce our RTM algorithm for imaging the obstacle with phaseless data. In section 3 we consider the resolution of our algorithm for imaging sound soft obstacles. In section 4 we extend our theoretical results to non-penetrable obstacles with the impedance boundary condition and penetrable obstacles.
In section 5 we report several numerical experiments to
show the competitive performance of our phaseless RTM algorithm.

\section{Reverse time migration method}
In this section we introduce the RTM method for inverse scattering problems with phaseless data.
Assume that there are $N_s$ emitters and $N_r$ receivers uniformly distributed respectively on $\Ga_s=\pa B_s$ and $\Ga_r=\pa B_r$, where $B_s, B_r$ are the disks of radius $R_s, R_r$ respectively. We denote by $\Om$ the sampling domain in which the obstacle is sought. We assume the obstacle $D\subset\Om$ and $\Om$ is inside $B_s,B_r$.

Let $u^i(x,x_s)=\Phi(x,x_s)$, where $\Phi(x,x_s)=\frac\i 4H^{(1)}_0(k|x-x_s|)$ is the fundamental solution of
the Helmholtz equation with the source at $x_s\in\Ga_s$, be the incident wave and $|u(x_r,x_s)|=|u^s(x_r,x_s)+u^i(x_r,x_s)|$ be the phaseless data received at $x_r\in\Ga_r$, where $u^s(x,x_s)$ is the solution to the problem \eqref{p1}-\eqref{p3} with $u^i(x,x_s)=\Phi(x,x_s)$. We additionally assume that $x_s\neq x_r$ for all $s=1,2,...,N_s, r=1,2,...,N_r$, to avoid the singularity of the incident field $u^i(x,x_s)$ at $x=x_r$. This assumption can be easily satisfied in practical applications.  In the following, without loss of generality, we assume $R_r=\tau R_s, \tau\ge1$.

Our RTM  algorithm consists of  back-propagating the corrected data:
\bee\label{De}
\De(x_r,x_s)=\frac{|u(x_r,x_s)|^2-|u^i(x_r,x_s)|^2}{u^i(x_r,x_s)}
\eee
into the domain using the fundamental solution $\Phi(x_r,z)$ and then computing the imaginary part of the cross-correlation between $u^i(z,x_s)$  and the back-propagated field.

\noindent
\begin{alg} {\sc (RTM for Phaseless data)}
Given the data $|u(x_r,x_s)|=|u^s(x_r,x_s)+u^i(x_r,x_s)|$ which is the measurement of the total field at $x_r\in\Ga_r$ when the point source is emitted at $x_s\in\Ga_s$, $s=1,\dots, N_s$, $r=1,\dots,N_r$. \\
$1^\circ$ Back-propagation: For $s=1,\dots,N_s$, compute the back-propagation field
\bee\label{sback}
v_b(z,x_s)=-\frac{2\pi R_r}{N_r}\sum^{N_r}_{r=1}\Phi(x_r,z)\De(x_r,x_s),\ \ \ \ \forall \ z\in\Om.
\eee
$2^\circ$ Cross-correlation: For $z\in\Om$, compute
\bee\label{scor1}
I(z)=-k^2\Im\left\{\frac{2\pi R_s}{N_s}\sum^{N_s}_{s=1} u^i(z,x_s)v_b(z,x_s)\right\}.
\eee
\end{alg}
It is easy to see that
\bee\label{scor2}
I(z)=
-k^2\Im\left\{\frac{(2\pi)^2 R_s R_r}{N_sN_r}\sum^{N_s}_{s=1}\sum^{N_r}_{r=1}\Phi(z,x_s)\Phi(x_r,z)\De(x_r,x_s)\right\},\ \ \forall \ z\in\Om.
\eee
This is the formula used in our numerical experiments in section 5. By letting $N_s,N_r\to\infty$, we know that \eqref{scor2} can be viewed as an approximation of the following continuous integral:
\bee\label{scor3}
\hat I(z)=-k^2\Im\int_{\Ga_s}\int_{\Ga_r} \Phi(z,x_s)\Phi(x_r,z)\De(x_r,x_s) ds(x_r)ds(x_s),\ \ \forall \ z\in\Om.
\eee

We remark that the above RTM imaging algorithm is the same as the RTM method in \cite{cch_a} except that the input data now is $\De(x_r,x_s)$ instead of $\overline{u^s(x_r,x_s)}$. Hence, the code of the RTM algorithm for imaging the obstacle with phaseless data requires only one line change from the code
of the RTM method for imaging the obstacle with full phase information.

\section{The resolution analysis}

In this section we study the resolution of the Algorithm 2.1. We first introduce some notation.  For any bounded domain $U\subset \R^2$ with Lipschitz boundary $\Ga$, let
$\|u\|_{H^1(U)}=(\|\na \phi\|_{L^2(U)}^2+d_U^{-2}\|\phi\|_{L^2(U)}^2)^{1/2}$ be the weighted $H^1(U)$ norm
and
$\|v\|_{H^{1/2}(\Ga)}=(d_U^{-1}\|v\|_{L^2(\Ga)}^2+|v|_{\frac 12,\Ga}^2)^{1/2}$ be the weighted $H^{1/2}(\Ga)$ norm,
where $d_U$ is the diameter of $U$ and
\ben
|v|_{\frac 12,\Ga}=\left(\int_\Ga\int_\Ga\frac{|v(x)-v(y)|^2}{|x-y|^2}ds(x)ds(y)\right)^{1/2}.
\een
By scaling argument and trace theorem we know that there exists a constant $C>0$ independent of $d_D$ such that for any $\phi\in C^1(\bar D)$,
\bee\label{d00}
\|\phi\|_{H^{1/2}(\Ga_D)}+\|\pa\phi/\pa\nu\|_{H^{-1/2}(\Ga_D)}\le C\max_{x\in\bar D}(|\phi(x)|+d_D|\na\phi(x)|).
\eee
The following stability estimate of the forward acoustic scattering problem is well-known \cite{colton-kress, mclean00}.

\begin{lem}{\label{lem:wp}}
Let~$g \in H^{1/2}(\Ga_D)$, then the scattering problem:
\be
& &\Delta w + k^2 w =0 \qquad\mbox{\rm in } \R^2\bks \bar{D}, \ \ \ \ w = g \ \ \ \ \mbox{\rm on } \Ga_D, \label{ha}\\
& &\sqrt{r}\left(\frac{\pa w}{\pa r}-\i k w\right)\to 0,\ \ \ \ \mbox{as }r\to\infty,\label{ha1}
\ee
admits a unique solution $w \in H^{1}_{\rm loc}(\R^2 \backslash \bar D)$. Moreover, there exists a constant $C>0$ such that $\|\pa w/\pa\nu\|_{H^{-1/2}(\Ga_D)}\le C\|g\|_{H^{1/2}(\Ga_D)}$.
\end{lem}

The far field pattern $w^\infty(\hat x)$, where $\hat x=x/|x|\in S^1=\{x\in\R^2:|x|=1\}$, of the solution of the scattering problem \eqref{ha}-\eqref{ha1} is defined as (cf. e.g. \cite[P. 67]{colton-kress}):
\bee\label{far}
w^\infty(\hat x)=\frac{e^{\i\frac\pi 4}}{\sqrt{8\pi k}}\int_{\Ga_D}\left[w(y)\frac{\pa e^{-\i k\hat x\cdot y}}{\pa\nu(y)}-
\frac{\pa w(y)}{\pa\nu(y)}e^{-\i k\hat x\cdot y}\right]ds(y).
\eee
It is well-known that for the scattering solution of \eqref{ha}-\eqref{ha1} (cf. e.g. \cite[Lemma 3.3]{cch_a})
\bee\label{y1}
-\,\Im\int_{\Ga_D}w\frac{\pa\bar w}{\pa\nu}ds=k\int_{S^1}|w^\infty(\hat x)|^2d\hat x.
\eee

Now we turn to the analysis of the imaging function $\hat I(z)$ in \eqref{scor3}. We first observe that
\bee\label{c1}
\De(x_r,x_s) = \overline{u^s(x_r,x_s)} + \frac{|u^s(x_r,x_s)|^2}{u^i(x_r,x_s)}+\frac{u^s(x_r,x_s)\overline{u^i(x_r,x_s)}}{u^i(x_r,x_s)}.
\eee
This yields
\be\label{c3}
\hat I(z)&=&-k^2\Im\int_{\Ga_s}\int_{\Ga_r} \Phi(z,x_s)\Phi(x_r,z)\overline{u^s(x_r,x_s)} ds(x_r)ds(x_s)\nonumber\\
& &-k^2\Im\int_{\Ga_s}\int_{\Ga_r} \Phi(z,x_s)\Phi(x_r,z)\frac{|u^s(x_r,x_s)|^2}{u^i(x_r,x_s)}ds(x_r)ds(x_s)\nonumber\\
& &-k^2\Im\int_{\Ga_s}\int_{\Ga_r} \Phi(z,x_s)\Phi(x_r,z)\frac{u^s(x_r,x_s)\overline{u^i(x_r,x_s)}}{u^i(x_r,x_s)} ds(x_r)ds(x_s).
\ee

The first term is the RTM imaging function with full phase information in \cite{cch_a} and thus can be analyzed by the argument there. Our goal now is to show the last two terms at the right hand side of \eqref{c3} are small. We start with the following lemma.

\begin{lem}\label{lem:4.1} We have $|u^s(x_r,x_s)|\le C(1+kd_D)^2(kR_r)^{-1/2}(kR_s)^{-1/2}$ for any $x_r\in\Ga_r,x_s\in\Ga_s$.
\end{lem}

\debproof We first recall the following estimates for Hankel functions \cite[(1.22)-(1.23)]{cg09}:
\bee\label{b1}
|H^{(1)}_0(t)|\le\left(\frac{2}{\pi t}\right)^{1/2},\ \ |H^{(1)}_1(t)|\le\left(\frac 2{\pi t}\right)^{1/2}+\frac 2{\pi t},\ \ \ \ \forall t>0.
\eee
By the integral representation formula, we have
\bee\label{b2}
u^s(x_r,x_s)=\int_{\Ga_D}\left(u^s(y,x_s)\frac{\pa \Phi(x_r,y)}{\pa\nu(y)}-\frac{\pa u^s(y,x_s)}{\pa\nu(y)}\Phi(x_r,y)\right)ds(y).
\eee
By \eqref{d00} we have
\ben
\|\Phi(x_r,\cdot)\|_{H^{1/2}(\Ga_D)}+\|\pa\Phi(x_r,\cdot)/\pa\nu\|_{H^{-1/2}(\Ga_D)}\le C(1+kd_D)(kR_r)^{-1/2}.
\een
The lemma follows now from Lemma \ref{lem:wp} and the fact that $u^i(y,x_s)=-\Phi(y,x_s)$ for $y\in\Ga_D$.
\finproof

\begin{lem}\label{lem:4.2} We have $|H^{(1)}_0(t)|\ge [2/(5\pi e)]|\ln t|$ for any $t\in (0,1)$.
\end{lem}

\debproof We use the following integral formula \cite{ob73, cg09}
\ben
H^{(1)}_0(t)=-\frac {2\i}{\pi}e^{\i t}\int^\infty_0\frac{e^{-r t}}{r^{1/2}(r-2\i)^{1/2}}dr,\ \ \ \ \forall t>0,
\een
where $\Re(r-2\i)^{1/2}>0$ for $r>0$. By the change of variable
\ben
|H^{(1)}_0(t)|\ge\frac 2\pi\Re\int^\infty_0\frac{e^{-s}}{s^{1/2}(s-2\i t)^{1/2}}ds
&=&\frac 2\pi\int^\infty_0\frac{e^{-s}}{s^{1/2}|s-2\i t|}\sqrt{\frac{|s-2\i t|+s}{2}}ds\\
&\ge&\frac 2\pi\int^\infty_0\frac{e^{-s}}{|s-2\i t|}ds\\
&\ge&\frac 2{5\pi}\int^1_t\frac{e^{-s}}{s}ds,
\een
where in the last inequality we have used $|s-2\i t|\le 5s$ for $s\ge t$.
This completes the proof by noticing that $\int^1_t s^{-1}e^{-s}ds\ge e^{-1}\int^1_t s^{-1}ds= e^{-1}|\ln t|$.
\finproof

The following lemma gives an estimate of the second term at the right hand side of \eqref{c3}.

\begin{lem}\label{lem:4.3} We have
\ben
\left| k^2\int_{\Ga_s}\int_{\Ga_r} \Phi(z,x_s)\Phi(x_r,z) \frac{|u^s(x_r,x_s)|^2}{u^i(x_r,x_s)}  ds(x_r)ds(x_s) \right|\le C(1+kd_D)^4(kR_s)^{-1/2}.
\een
\end{lem}

\debproof Denote $\Om_k=\{(x_r,x_s)\in\Ga_r\times\Ga_s: |x_r-x_s|<1/(2k)\}$.
Let $x_r=R_r(\cos\theta_r,\sin\theta_r)$, $x_s=R_s(\cos\theta_s,\sin\theta_s)$, where $\theta_r,\theta_s\in [0,2\pi]$. Since $R_r=\tau R_s$, $\tau\ge1$, we obtain easily
\bee\label{y2}
|x_r-x_s|=R_s\sqrt{1+\tau^2-2\tau\cos|\theta_r-\theta_s|}\ge 2R_s\sqrt\tau\sin\frac{|\theta_r-\theta_s|}2
\eee
by Cauchy-Schwarz inequality. Now for $(x_r,x_s)\in\Om_k$, we have then either $\frac{|\theta_r-\theta_s|}2\le\theta_0$ or $\pi-\theta_0\le\frac{|\theta_r-\theta_s|}2\le\pi$, where $\theta_0=\arcsin\frac{1}{4kR_s\sqrt\tau}\in (0,\pi/2)$.

By \eqref{b1} and Lemma \ref{lem:4.1},
\be\label{c4}
& &\left|\int\hskip-5pt\int_{\Om_k}\Phi(z,x_s)\Phi(x_r,z) \frac{|u^s(x_r,x_s)|^2}{u^i(x_r,x_s)}  ds(x_r)ds(x_s)\right|\nonumber\\
&\le&C(1+kd_D)^4(kR_s)^{-3/2}(kR_r)^{-3/2}\int\hskip-5pt\int_{\Om_k}|\ln (k|x_r-x_s|)|ds(x_r)ds(x_s).
\ee
By \eqref{y2} we have
\ben
\int\hskip-5pt\int_{\Om_k}|\ln (k|x_r-x_s|)|ds(x_r)ds(x_s)
&\le&-\int\hskip-5pt\int_{\Om_k}\ln\left(2kR_s\sqrt\tau\sin\frac{|\theta_r-\theta_s|}{2}\right)d\theta_rd\theta_s\\
&\le&2\pi R_rR_s\left|\int^{\theta_0}_0\ln (2kR_s\sqrt\tau\sin t)dt\right|\\
&\le&CR_rR_s,
\een
where we have used integration by parts in obtaining the last inequality. Substituting the above estimate into
\eqref{c4} we obtain
\bee\label{c5}
\left|\int\hskip-5pt\int_{\Om_k}\Phi(z,x_s)\Phi(x_r,z) \frac{|u^s(x_r,x_s)|^2}{u^i(x_r,x_s)}  ds(x_r)ds(x_s)\right|
\le Ck^{-2}(1+kd_D)^4(kR_s)^{-1}.
\eee
Next we estimate the integral in $\Ga_r\times\Ga_s\backslash\bar\Om_k$. Since $t|H_{0}^{(1)}(t)|^2$ is an increasing function of $t >0$ \cite[p. 446]{waston}, we have
for $(x_r,x_s)\in \Ga_r\times\Ga_s\backslash\bar\Om_k$, $|x_r-x_s|\ge 1/(2k)$, and thus
\ben
|x_r-x_s||u^i(x_r,x_s)|^2\ge\frac{1}{32} k^{-1}\left|H_{0}^{(1)}\left(\frac 12\right)\right|^2=Ck^{-1},
\een
which implies by using Lemma \ref{lem:4.1} and \eqref{b1} again that
\bee\label{c6}
\left|\int\hskip-5pt\int_{\Ga_r\times\Ga_s\backslash\bar\Om_k}\Phi(z,x_s)\Phi(x_r,z) \frac{|u^s(x_r,x_s)|^2}{u^i(x_r,x_s)} ds(x_r)ds(x_s) \right| \le Ck^{-2}(1+kd_D)^4(kR_s)^{-1/2}.
\eee
This completes the proof by combining the above estimate with \eqref{c5}.
\finproof

Now we turn to the estimation of the third term at the right hand side of \eqref{c3}. Denote $\de=(kR_s)^{-1/2}$
and $\Theta_\de:=\{(\theta_r,\theta_s)\in (0,2\pi)^2: |\theta_r-\theta_s\pm m\pi|<\de,\,m=0,1,2\}$ and $Q_\de:=\{(x_r,x_s)\in
\Ga_r\times\Ga_s: (\theta_r,\theta_s)\in\Theta_\de\}$.

\begin{lem}\label{lem:4.4} We have
\ben
\left|k^2\int\hskip-5pt\int_{Q_\de} \Phi(z,x_s)\Phi(x_r,z) u^s(x_r,x_s) \frac{\overline{u^i(x_r,x_s)}}{u^i(x_r,x_s)} ds(x_r)ds(x_s)\right|\le C(1+kd_D)^2(kR_s)^{-1/2}.
\een
\end{lem}

\debproof The proof follows easily from Lemma \ref{lem:4.1} and \eqref{b1} and the fact that $|Q_\de|\le CR_rR_s(kR_s)^{-1/2}$.
\finproof

To estimate the integral in $\Ga_r\times\Ga_s\backslash \bar Q_\de$, we recall first the following useful mixed reciprocity relation \cite{k97}, \cite[P.40]{p01}.

\begin{lem}\label{lem:4.5}
Let $\ga_m=e^{\i\frac\pi 4}/\sqrt{8\pi k}$. Then $u^\infty_\ps(d,x_s)=\ga_m u^s(x_s,-d)$ for any $x_s\in\R^2\backslash\bar D,d\in S^1$, where $u^\infty_\ps(d,x_s)$ is the far field pattern in the direction $d$ of the scattering solution of \eqref{p1}-\eqref{p3} with $u^i(x)=\Phi(x,x_s)$ and $u^s(x,d)$ is the scattering solution of \eqref{p1}-\eqref{p3} with the incident plane wave $u^i(x)=e^{\i kx\cdot d}$.
\end{lem}

\begin{lem}\label{lem:4.7} Let $u^\infty_\pl(\hat x_s,-\hat x_r)$ be the far field pattern in the direction $\hat x_s$ of the scattering solution
of the Helmholtz equation with the incident plane wave $e^{-\i k\hat x_r\cdot x}$. Then
$|u^\infty_\pl(\hat x_s,-\hat x_r)|+(kd_D)^{-1}|\pa u^\infty_\pl(\hat x_s,-\hat x_r)/\pa\theta_s|\le Ck^{-1/2}(1+kd_D)^2$.
\end{lem}

\debproof The proof follows from the definition of the far field pattern in \eqref{far} with $g(y)=-e^{-\i k\hat x_r\cdot y}$ on $\Ga_D$, Lemma \ref{lem:wp}, and \eqref{d00}. Here we omit the details.
\finproof

The following lemma is essentially proved in \cite[Theorem 2.5]{colton-kress}.

\begin{lem}\label{lem:far} For any $x\in\R^2\backslash\bar D$, the solution of the scattering problem \eqref{ha}-\eqref{ha1} satisfies the asymptotic behavior:
\ben
w(x)=\frac{e^{\i k|x|}}{\sqrt{|x|}}w^\infty(\hat x)+\ga(x),
\een
where $|\ga(x)|\le C(1+kd_D)^3(k|x|)^{-3/2}\|g\|_{H^{1/2}(\Ga_D)}$.
\end{lem}

\begin{proof} First we have the following integral representation formula
\bee\label{r1}
w(x)=\int_{\Ga_D}\left[w(y)\frac{\pa \Phi(x,y)}{\pa\nu(y)}-
\Phi(x,y)\frac{\pa w(y)}{\pa\nu(y)}\right]ds(y),\ \ \ \ \forall x\in\R^2\backslash\bar D.
\eee
By the asymptotic formulae of Hankel functions \cite[P.197]{waston}, for $n=1,2$,
\bee\label{r2}
H_n^{(1)}(t)=\left(\frac 2{\pi t}\right)^{1/2}e^{\i (kt-\frac {n\pi}2-\frac\pi 4)}+R_n(t),\ \ |R_n(t)|\le Ct^{-3/2},\ \ \forall t>0,
\eee
and the simple estimate $|k|x-y|-k(|x|-\hat x\cdot y)|\le C(k|y|)^2(k|x|)^{-1}$ for any $y\in \Ga_D$, $x\in\R^2\backslash\bar D$, we have
\be
\Phi(x,y)&=&\frac{e^{\i\frac \pi 4}}{\sqrt{8\pi k}}\frac{e^{\i k|x|}}{\sqrt{|x|}}e^{-\i k\hat x\cdot y}+\ga_0(x,y),\label{r3}\\
\frac{\pa\Phi(x,y)}{\pa\nu(y)}&=&\frac{e^{\i\frac \pi 4}}{\sqrt{8\pi k}}\frac{e^{\i k|x|}}{\sqrt{|x|}}\frac{\pa e^{-\i k\hat x\cdot y}}{\pa\nu(y)}+\ga_1(x,y),\label{r4}
\ee
where $|\ga_0(x,y)|+k^{-1}|\ga_1(x,y)|\le C(1+k|y|)^2(k|x|)^{-3/2}$ for some constant $C$ independent of $k$ and $D$. The proof completes by
inserting \eqref{r3}-\eqref{r4} into \eqref{r1} and using Lemma \ref{lem:wp} and \eqref{d00}. Here we omit the details.
\end{proof}

We also need the following slight generalization of Van der Corput lemma for the oscillatory integral \cite[P.152]{grafakos}.

\begin{lem}\label{lem:4.6}
For any $-\infty<a<b<\infty$, for every real-valued $C^2$ function $u$ that satisfies $|u'(t)|\ge 1$ for $t\in (a,b)$. Assume that $a=x_0<x_1<\cdots<x_N=b$ is a division of $(a,b)$ such that $u'$ is monotone in each
interval $(x_{i-1},x_i)$, $i=1,\cdots, N$. Then for any function $\phi$ defined on $(a,b)$ with integrable derivative, and for any $\lambda>0$,
\ben
\left|\int^b_a e^{\i\lambda u(t)}\phi(t)dt\right|\le
(2N+2)\lambda^{-1}\left[|\phi(b)|+\int^b_a|\phi'(t)|dt\right].
\een
\end{lem}

\debproof  By integration by parts we have
\ben
\int^b_ae^{\i\lam u(t)}dt=\left[\frac{e^{\i\lam u(t)}}{\i\lam u'(t)}\right]^b_a-\int^b_ae^{\i\lam u(t)}\frac d{dt}\left(\frac 1{\i\lam u'(t)}\right)dt.
\een
Since $u'$ is monotone in each interval $(x_{i-1},x_i)$, $i=1,\cdots, N$, and $|u'(t)|\ge 1$ in $(a,b)$, we have
\ben
\left|\int^b_ae^{\i\lam u(t)}\frac d{dt}\left(\frac 1{\i\lam u'(t)}\right)dt\right|\le\sum^N_{i=1}\lam^{-1}\left|\int^{x_i}_{x_{i-1}}\frac d{dt}
\left(\frac 1{u'(t)}\right)dt\right|\le 2N\lam^{-1},
\een
which implies $|\int^b_ae^{\i\lam u(t)}dt|\le (2N+2)\lam^{-1}$. For the general case, we denote $F(t)=\int_a^te^{\i\lam u(s)}ds$ and use integration by parts to obtain
\ben
\int^b_a\phi(t)e^{\i\lam u(t)}dt=\phi(b)F(b)-\int^b_aF(t)\phi'(t)dt.
\een
This completes the proof by using $|F(t)|\le (2N+2)\lam^{-1}$. \finproof

\begin{lem}\label{lem:4.8} We have
\ben
& &\left| k^2\,\int\hskip-5pt\int_{\Ga_r\times\Ga_s\backslash\bar Q_\de} \Phi(z,x_s)\Phi(x_r,z) u^s(x_r,x_s) \frac{\overline{u^i(x_r,x_s)}}{u^i(x_r,x_s)} ds(x_r)ds(x_s)\right|\\
& &\qquad\le C(1+kd_D)^3(kR_s)^{-1/2}+C(1+k|z|)^2(kR_s)^{-1}.
\een
\end{lem}

\debproof We first observe that for $(x_r,x_s)\in \Ga_r\times\Ga_s\backslash \bar Q_\de$, we have
\ben
k|x_r-x_s|\ge 2kR_s\sqrt\tau\left|\sin\frac{\theta_r-\theta_s}{2}\right|\ge 2kR_s\sqrt\tau\sin\frac\de 2\ge \frac 12(kR_s)^{1/2}\sqrt\tau,
\een
where we have used the fact that $\sin t\ge t/2$ for $t\in (0,\pi/2)$. Thus by \eqref{r2} we obtain
\bee\label{d3}
\frac{\overline{u^i(x_r,x_s)}}{u^i(x_r,x_s)}
=e^{-2\i k|x_r-x_s|+\i\frac\pi 2}+\rho_0(x_r,x_s),
\eee
where $|\rho_0(x_r,x_s)|\le C(kR_s)^{-1/2}$. Similar to \eqref{r3} we have
\be
\Phi(z,x_s)&=&\frac{e^{\i\frac\pi 4}}{\sqrt{8\pi k}}\frac{e^{\i kR_s}}{\sqrt{R_s}}e^{-\i k\hat x_s\cdot z}+\rho_1(z,x_s),\label{d5}\\
\Phi(z,x_r)&=&\frac{e^{\i\frac\pi 4}}{\sqrt{8\pi k}}\frac{e^{\i kR_r}}{\sqrt{R_r}}e^{-\i k\hat x_r\cdot z}+\rho_1(z,x_r),\label{d6}
\ee
where $|\rho_1(z,x_s)|\le C(1+k|z|)^2(kR_s)^{-3/2}$, $|\rho_1(z,x_r)|\le C(1+k|z|)^2(kR_r)^{-3/2}$.
Next by Lemma \ref{lem:far}, the mixed reciprocity in Lemma \ref{lem:4.5}, and \eqref{d00}
we have
\be\label{d4}
u^s(x_r,x_s)&=&\frac{e^{\i kR_r}}{\sqrt{R_r}}u^\infty_\ps(\hat x_r,x_s)+\rho_2(x_r,x_s)\nonumber\\
&=&\frac{e^{\i kR_r}}{\sqrt{R_r}}\ga_mu^s(x_s,-\hat x_r)+\rho_2(x_r,x_s)\nonumber\\
&=&\frac{e^{\i k(R_r+R_s)}}{\sqrt{R_rR_s}}\ga_mu^\infty_\pl(\hat x_s,-\hat x_r)+\rho_2(x_r,x_s)+\rho_3(x_r,x_s),
\ee
where $u^\infty_\pl(\hat x_s,-\hat x_r)$ is the far field pattern of the scattering solution of the Helmholtz equation with the incident plane wave $u^i=e^{-\i k\hat x_r\cdot x}$ and
\ben
|\rho_2(x_r,x_s)|&\le&C(1+kd_D)^3(kR_r)^{-3/2}\|\Phi(\cdot,x_s)\|_{H^{1/2}(\Ga_D)}\\
&\le&C(1+kd_D)^4(kR_r)^{-3/2}(kR_s)^{-1/2},\\
|\rho_3(x_r,x_s)|&\le&C(1+kd_D)^3(kR_r)^{-1/2}(kR_s)^{-3/2}\|e^{-\i k\hat x_r\cdot x}\|_{H^{1/2}(\Ga_D)}\\
&\le&C(1+kd_D)^4(kR_r)^{-1/2}(kR_s)^{-3/2}.
\een
Combining \eqref{d3}-\eqref{d4} we have
\be\label{d8}
&&k^2\int\hskip-5pt\int_{\Ga_r\times\Ga_s\backslash\bar Q_\de} \Phi(z,x_s)\Phi(x_r,z) u^s(x_r,x_s) \frac{\overline{u^i(x_r,x_s)}}{u^i(x_r,x_s)} ds(x_r)ds(x_s)\nonumber \\
&=&k^2R_sR_s\int\hskip-5pt\int_{(0,2\pi)^2\backslash\bar\Theta_\de}\Phi(z,x_s)\Phi(x_r,z) u^s(x_r,x_s) \frac{\overline{u^i(x_r,x_s)}}{u^i(x_r,x_s)} d\theta_rd\theta_s\nonumber\\
&=&-\frac{\ga_mk}{8\pi}e^{2\i k(R_r+R_s)}\int\hskip-5pt\int_{(0,2\pi)^2\backslash\bar\Theta_\de}\phi(\theta_r,\theta_s)e^{-2\i k|x_r-x_s|}d\theta_rd\theta_s+\rho(z),
\ee
where $\phi(\theta_r,\theta_s)=u^\infty_\pl(\hat x_s,-\hat x_r)e^{-\i k(x_s+x_r)\cdot z}$ and by Lemma \ref{lem:4.1}
\ben
|\rho(z)|&\le&C(1+kd_D)^2(kR_s)^{-1/2}+C(1+kd_D)^2(kR_s)^{-1/2}+C(1+k|z|)^2(kR_s)^{-1},\\
&\le&C(1+kd_D)^2(kR_s)^{-1/2}+C(1+k|z|)^2(kR_s)^{-1},
\een
where the second inequality follows from the fact that we are interested in the situation when $kR_s$ is sufficiently large such that $(1+kd_D)^2(kR_s)^{-1/2}\ll 1$. Now direct calculation shows that
\be\label{d7}
& &\int\hskip-5pt\int_{(0,2\pi)^2\backslash\bar\Theta_\de}\phi(\theta_r,\theta_s)e^{-2\i k|x_r-x_s|}d\theta_rd\theta_s\nonumber\\
&=&\int_{0}^{\de}\int_{(\theta_r+\de,\theta_r+\pi-\de)\cup (\theta_r+\pi+\de,\theta_r+2\pi-\de)}\phi(\theta_r,\theta_s)e^{-2\i k|x_r-x_s|}d\theta_sd\theta_r\nonumber\\
&+&\int_{\de}^{\pi-\de}\int_{(0,\theta_r-\de)\cup (\theta_r+\de,\theta_r+\pi-\de)\cup (\theta_r+\pi+\de,2\pi)}\phi(\theta_r,\theta_s)e^{-2\i k|x_r-x_s|}d\theta_sd\theta_r\nonumber\\
&+&\int_{\pi-\de}^{\pi+\de}\int_{(\theta_r-\pi+\de,\theta_r-\de)\cup (\theta_r+\de,\theta_r+\pi-\de)}\phi(\theta_r,\theta_s)e^{-2\i k|x_r-x_s|}d\theta_sd\theta_r\nonumber\\
&+&\int_{\pi+\de}^{2\pi-\de}\int_{(0,\theta_r-\pi-\de)\cup (\theta_r-\pi+\de,\theta_r-\de)\cup (\theta_r+\de,2\pi)}\phi(\theta_r,\theta_s)e^{-2\i k|x_r-x_s|}d\theta_sd\theta_r\nonumber\\
&+&\int_{2\pi-\de}^{2\pi}\int_{(\theta_r-2\pi+\de,\theta_r-\pi-\de)\cup(\theta_r-\pi+\de,\theta_r-\de)}\phi(\theta_r,\theta_s)e^{-2\i k|x_r-x_s|}d\theta_sd\theta_r\nonumber\\
&:=&{\rm I}_1+\cdots+{\rm I}_5.
\ee
By Lemma \ref{lem:4.7} and $\de=(kR_s)^{-1/2}$ we have $|{\rm I}_1+{\rm I}_3+{\rm I}_5|\le Ck^{-1/2}(1+kd_D)^2(kR_s)^{-1/2}$.

We will use Lemma \ref{lem:4.6} to estimate ${\rm I_2}$ and ${\rm I}_4$. For that purpose, denote by $v(\theta_s)=-\sqrt{1+\tau^2-2\tau\cos(\theta_r-\theta_s)}$. We have $v'(\theta_s)=\tau\sin(\theta_s-\theta_r)/v(\theta_s)$ and thus $|v'(\theta_s)|\ge\tau|\sin\de|/|v(\theta_s)|\ge\frac\tau{1+\tau}\frac\de 2\ge\de/4=\frac 14(kR_s)^{-1/2}$ for $(\theta_r,\theta_s)\in
\Ga_r\times\Ga_s\backslash\bar\Theta_\de$. Moreover, $v''(\theta_s)=-\tau^2(\cos(\theta_s-\theta_r)-\tau)(\cos(\theta_s-\theta_r)-\tau^{-1})/v(\theta_s)^3$ which implies $v'(\theta_s)$ is piecewise monotone in $(0,2\pi)$ for any fixed $\theta_r\in (0,2\pi)$ since $\tau\ge 1$.

Now since $-2\i k|x_r-x_s|=2\i kR_s v(\theta_s)$, we obtain
by Lemma \ref{lem:4.6} and Lemma \ref{lem:4.7} that for $\theta_r\in (\de,\pi-\de)$,
\ben
\left|\int_0^{\theta_r-\de}\phi(\theta_r,\theta_s)e^{-2\i k|x_r-x_s|}d\theta_s\right|\le Ck^{-1/2}(1+kd_D)^3(kR_s)^{-1/2}.
\een
The other integrals in ${\rm I}_2$ and ${\rm I}_4$ can be estimated similarly to obtain
\ben
|{\rm I}_2|+|{\rm I}_4|\le Ck^{-1/2}(1+kd_D)^3(kR_s)^{-1/2}.
\een
This completes the proof by \eqref{d8}.
\finproof

The following theorem is the main result of this section.

\begin{thm}\label{res3}
For any ~$z\in\Om$, let $\psi(x,z)$ be the scattering solution to the problem:
\bee\label{ps1}
\De\psi(x,z)+k^2\psi(x,z)=0\ \ \ \ \mbox{\rm in } \R^2  \bks \bar D,\ \ \ \
\psi(x,z)=-\Im \Phi(x,z) \ \ \mbox{\rm on } \Ga_D.
\eee
Then if the measured field $|u(x_r,x_s)|=|u^s(x_r,x_s)+u^i(x_r,x_s)|$ with $u^s(x,x_s)$ satisfying the problem \eqref{p1}-\eqref{p3} with the incident field $u^i(x,x_s)=\Phi(x,x_s)$, we have
\ben
\hat I(z)= k\int_{S^1}|\psi^\infty(\hat x,z)|^2d\hat x+w_{\hat I}(z),\ \ \ \ \forall z\in\Om,
\een
where $|w_{\hat I}(z)|\le C(1+kd_D)^4(kR_s)^{-1/2}+C(1+kd_D)^2(1+k|z|)^2(kR_s)^{-1}$.
\end{thm}

\debproof By \eqref{c3}, Lemma \ref{lem:4.4} and Lemma \ref{lem:4.8} we are left to show that
\ben
-k^2\,\Im\int_{\Ga_s}\int_{\Ga_r} \Phi(z,x_s)\Phi(x_r,z)   \overline{ u^s(x_r,x_s) }  ds(x_r)ds(x_s) = k\int_{S^1}|\psi^\infty(\hat x,z)|^2d\hat x +\zeta(z),
\een
with $|\zeta(z)|\le C(1+kd_D)^2(1+kd_D+k|z|)^2(kR_s)^{-1}$. This can be done by a similar argument as that in \cite[Theorem 3.2]{cch_a}. For the
sake of completeness, we include a sketch of the proof here.

By the corollary of Helmholtz-Kirchhoff identity in \cite[Lemma 3.2]{cch_a}, for any $z,y\in\Om$,
\bee
k\int_{\Ga_r}\Phi(x_r,z)\overline{\Phi(x_r,y)}ds(x_r)=\Im\Phi(z,y)+w_r(z,y),\label{s1}
\eee
where $|w_r(z,y)|+k^{-1}|w_r(z,y)|\le C(1+k|y|+k|z|)^2(kR_r)^{-1}$. Now by the integral representation formula
\ben
u^s(x_r,x_s)=\int_{\Ga_D}\left(u^s(y,x_s)\frac{\pa \Phi(x_r,y)}{\pa\nu(y)}-\frac{\pa u^s(y,x_s)}{\pa\nu(y)}\Phi(x_r,y)\right)ds(y),
\een
we have
\ben
& &k\int_{\Ga_r} \Phi( x_r,z)\overline{u^s(x_r,x_s)}ds(x_r)\\
&=&\int_{\Ga_D}\Big[ \overline{u^s(y,x_s)}\frac{\pa \Im \Phi(z,y)}{\pa \nu(y)}-\frac{\pa \overline{u^s(y,x_s)}}{\pa\nu(y)}\Im \Phi(z,y)\Big]ds(y) + \zeta_1(z,x_s),
\een
where by \eqref{d00}
\ben
|\zeta_1(z,x_s)|&=&\left|\int_{\Ga_D}\Big[ \overline{u^s(y,x_s)}\frac{\pa w_r(z,y)}{\pa \nu(y)}-\frac{\pa \overline{u^s(y,x_s)}}{\pa\nu(y)}w_r(z,y)\Big]ds(y)\right|\\
&\le&C(1+kd_D)^2(1+kd_D+k|z|)^2(kR_s)^{-3/2}.
\een
By the definition of the imaging function $\hat I(z)$, we have then
\bee\label{cor4_ha}
\hat I(z)=-\Im\int_{\Ga_D}\Big[v_s(y,z) \frac{\pa \Im \Phi(z,y)}{\pa \nu(y)} -\frac{\pa v_s(y,z)}{\pa\nu(y)} \Im \Phi(z,y)\Big]ds(y)+ \zeta_2(z),
\eee
where $v_s(y,z)= k\int_{\Ga_s}\Phi(x_s,z)\overline{u^s(y,x_s)}ds(x_s)$ and
\be\label{s2}
|\zeta_2(z)|&=&k\left|\int_{\Ga_s}\Phi(x_s,z)\zeta_1(z,x_s)ds(x_s)\right|\nonumber\\
&\le&C(1+kd_D)^2(1+kd_D+k|z|)^2(kR_s)^{-1}.
\ee
Taking the complex conjugate we get
\ben
\overline{v_s(y,z)}= k\int_{\Ga_s}\overline{\Phi(x_s,z)}u^s(y,x_s)ds(x_s).
\een
Therefore, $\overline{v_s(y,z)}$ can be viewed as the weighted superposition of $u^s(y,x_s)$. Then $\overline{v_s(y,z)}$ satisfies the Helmholtz  equation
\ben
\De_y\overline{v_s(y,z)}+k^2\overline{v_s(y,z)}=0\ \ \ \ \mbox{in }\R^2\bks\bar D.
\een
On the boundary of the obstacle $\Ga_D$, we have
\ben
\overline{v_s(y,z)}
&=& k\int_{\Ga_s}\overline{\Phi(x_s,z)} u^s(y,x_s)ds(x_s)\\
&=&- k\int_{\Ga_s}\overline{\Phi(x_s,z)}\Phi(y,x_s)ds(x_s)\\
&=&- \Im \Phi(z,y) - w_s(z,y) ,\ \ \ \ \forall y\in\Ga_D,
\een
where, similar to \eqref{s1}, $|w_s(z,y)|+k^{-1}|w_s(z,y)|\le C(1+k|y|+k|z|)^2(kR_s)^{-1}$.
By the definition of $\psi$ we have
\ben
\hat{I}(z) &=&-\Im\int_{\Ga_D}\Big[\overline{\psi(y,z)} \frac{\pa \Im \Phi(z,y)}{\pa \nu(y)} -\frac{\overline{\pa\psi(y,z)}}{\pa\nu(y)} \Im \Phi(z,y)\Big]ds(y)+\zeta_3(z), \\
 &=& -\Im\int_{\Ga_D}\frac{\overline{\pa\psi(y,z)}}{\pa\nu(y)} \psi(y,z)ds(y)+\zeta_3(z).
\een
where we have used the boundary condition of $\psi$ on $\Ga_D$ in the second inequality and
\ben
\zeta_3(z)=\zeta_2(z)+\Im\int_{\Ga_D}\Big[\overline{\phi(y,z)} \frac{\pa \Im \Phi(z,y)}{\pa \nu(y)} -\frac{\overline{\pa\phi(y,z)}}{\pa\nu(y)} \Im \Phi(z,y)\Big]ds(y).
\een
Here $\phi(y,z)$ is the solution of scattering problem \eqref{ha}-\eqref{ha1} with the boundary condition $g=w_s(z,y)$. Now by \eqref{d00} and \eqref{s2} we then obtain
\ben
|\zeta_3(z)|\le C(1+kd_D)^2(1+kd_D+k|z|)^2(kR_s)^{-1}.
\een
Now the theorem is proved by using \eqref{y1}.
\finproof

We remark that $\psi(x,z)$ is the scattering solution of the Helmholtz equation  with the incoming field
$J_0(k|x-z|)$.
It is well-known that $J_0(t)$ peaks at $t=0$ and decays like $t^{-1/2}$ away from the origin. The source of the problem \eqref{ps1} will peak at the boundary of the scatterer $D$ and becomes small when $z$ moves away from $\pa D$. Thus we expect that the imaging function $\hat I(z)$ will have a large contrast at the boundary of the scatterer $D$ and decay outside the boundary $\pa D$. This is indeed observed in our numerical experiments.

\section{Extensions}

In this section we consider briefly the imaging of the penetrable and impedance non-penetrable obstacles with phaseless data. We first consider the imaging of impedance non-penetrable obstacles with the phaseless data, in which case, the measured phaseless total field $|u(x_r,x_s)|=|u^s(x_r,x_s)+u^i(x_r,x_s)|$, where $u^s(x,x_s)$ is the radiation solution of the following problem:
\be
& & \Delta u^s + k^2 u^s =0 \qquad \mbox{in } \R^2\bks\bar D, \label{h1}\\
& & \frac{\pa u^s}{\pa\nu}+\i k\eta(x)u^s=-\frac{\pa u^i}{\pa\nu}-\i k\eta(x)u^i \ \ \ \ \mbox{ on } \Ga_D. \label{h2}
\ee
Here $\eta(x)>0$ is the impedance function. The well-posedness of the problem  \eqref{h1}-\eqref{h2} is well-known \cite{ccm01, leis}. By modifying
the argument in section 3 and \cite[Theorem 3.2]{cch_a} we can show the following theorem whose proof is omitted.

\begin{thm}\label{res_dir}
 For any $z\in\Om$, let $\psi(x,z)$ be the radiation solution of the problem
\ben
& &\De\psi(x,z)+k^2\psi(x,z)=0\ \ \ \ \mbox{\rm in }\R^2\bks\bar D,\\
    & &    \frac{\pa \psi(x,z)}{\pa\nu}+\i k\eta(x)\psi(x,z)=- \frac{\pa \Im\Phi(x,z)}{\pa\nu}-\i k\eta(x)  \Im\Phi(x,z) \ \ \ \ \mbox{\rm on }\Ga_D.
\een
Then if the measured field $|u(x_r,x_s)|=|u^s(x_r,x_s)+u^i(x_r,x_s)|$ with $u^s(x,x_s)$ satisfying \eqref{h1}-\eqref{h2}, we have, for any $z\in\Om$,
\ben
\hat I(z)=k\int_{S^1}|\psi^\infty(\hat x,z)|^2d\hat x+k\int_{\Ga_D}\eta(x)\left|\psi(x,z)+\Im\Phi(x,z)\right|^2d\hat{x}+w_{\hat I}(z),
\een
where $|w_{\hat I}(z)|\le C(1+kd_D)^4(kR_s)^{-1/2}+C(1+kd_D)^2(1+k|z|)^2(kR_s)^{-1}$.
\end{thm}

For penetrable obstacles,  the measured total field $|u(x_r,x_s)|=|u^s(x_r,x_s)+u^i(x_r,x_s)|$, where $u^s(x,x_s)$ is the radiation solution of the following problem
\bee
\Delta u^s + k^2n(x)u^s = -k^2(n(x)-1)u^i(x,x_s) \ \ \mbox{in } \R^2 \label{hh1}
\eee
with $n(x)\in L^\infty(\R^2)$ being a positive function which is equal to $1$ outside the scatterer $D$.
The well-posedness of the problem under some condition on $n(x)$ is known \cite{zhang94}. By modifying the argument in section 3 and in \cite[Theorem 3.1]{cch_a}, the following theorem can be proved. Here we omit the details.

\begin{thm}\label{res_penetrable}
 For any $z\in\Om$, let $\psi(x,z)$ be the radiation solution of the problem
\be
& &\Delta\psi+k^2n(x)\psi=-k^2(n(x)-1)\Im\Phi(x,z)\ \ \ \ \mbox{\rm in }\R^2.\label{ps0}
\ee
Then if the measured field $|u(x_r,x_s)|=|u^s(x_r,x_s)+u^i(x_r,x_s)|$ with $u^s(x,x_s)$ satisfying \eqref{hh1}, we have
\ben
\hat I(z)=k\int_{S^1}|\psi^\infty(\hat x,z)|^2d\hat x+w_{\hat I}(z)\ \ \ \ \forall z\in\Om,
\een
where $|w_{\hat I}(z)|\le C(1+kd_D)^4(kR_s)^{-1/2}+C(1+kd_D)^2(1+k|z|)^2(kR_s)^{-1}$.
\end{thm}

We remark that for the penetrable scatterers, $\psi(x,z)$ is again the scattering solution with the incoming field $\Im\Phi(x,z)$. Therefore we again expect the imaging function $\hat I(z)$ will have contrast on the boundary of the scatterer and decay outside the scatterer.

\section{Numerical examples}

In this section, we show several numerical experiments to illustrate the effectiveness of our RTM algorithm
with phaseless data in this paper. To synthesize the scattering data we compute the solution $u(x,x_s)$ of the scattering problem (\ref{p1})-(\ref{p3}) by standard Nystr\"{o}m's methods \cite{colton-kress}. The boundary integral equations on $\Ga_D$ are solved on a uniform mesh over the boundary with ten points per probe wavelength. The boundaries of the obstacles used in our numerical experiments are parameterized as follows,
where $\theta\in [0,2\pi]$,
\ben
\mbox{Kite: }&&x_1=\cos(\theta) + 0.65\cos(2\theta) - 0.65,\ \ x_2=1.5 \sin (\theta),\\
\mbox{$p$-leaf: }&&r(\theta)=1+0.2\cos(p\theta),\\
\mbox{Peanut: }&& x_1=\cos(\theta) + 0.2\cos(3\theta), \ \ x_2 = \sin(\theta) + 0.2\sin(3\theta),\\
\mbox{Rounded-square: } && x_1=\cos^3(\theta) + \cos(\theta) , \ \ x_2=\sin^3(\theta) + \sin(\theta).
\een

The sources $x_s$, $s=1,\cdots, N_s$, and the receivers $x_r$, $x_r=1,\cdots,N_r$, are uniformly distributed on $\Ga_{s}$
and $\Ga_r$, that is, $x_s=R_s(\cos \theta_s,\sin\theta_s),\theta_s=\frac{2\pi}{N_s}(s-1),s=1,2,...,N_s$,
and $x_r=R_r(\cos \theta_r,\sin\theta_r),\theta_r=\frac{2\pi}{N_r}(r-1)+\frac{\pi}{N_r},r=1,2,...,N_r$, so that $x_r\not= x_s$.

\begin{exmp}\label{ex1}
We consider the imaging of sound soft obstacles including a circle, a peanut, a kite and a rounded-square. The imaging domain is $\Om=(-3,3)\times(-3,3)$ with the sampling mesh $201\times201$. The probe wave wavenumber $k=4\pi$, $N_s=N_r=128$, and $R_s=R_r=10$.
\end{exmp}
\begin{figure}
    \centering
    \includegraphics[width=0.4\textwidth]{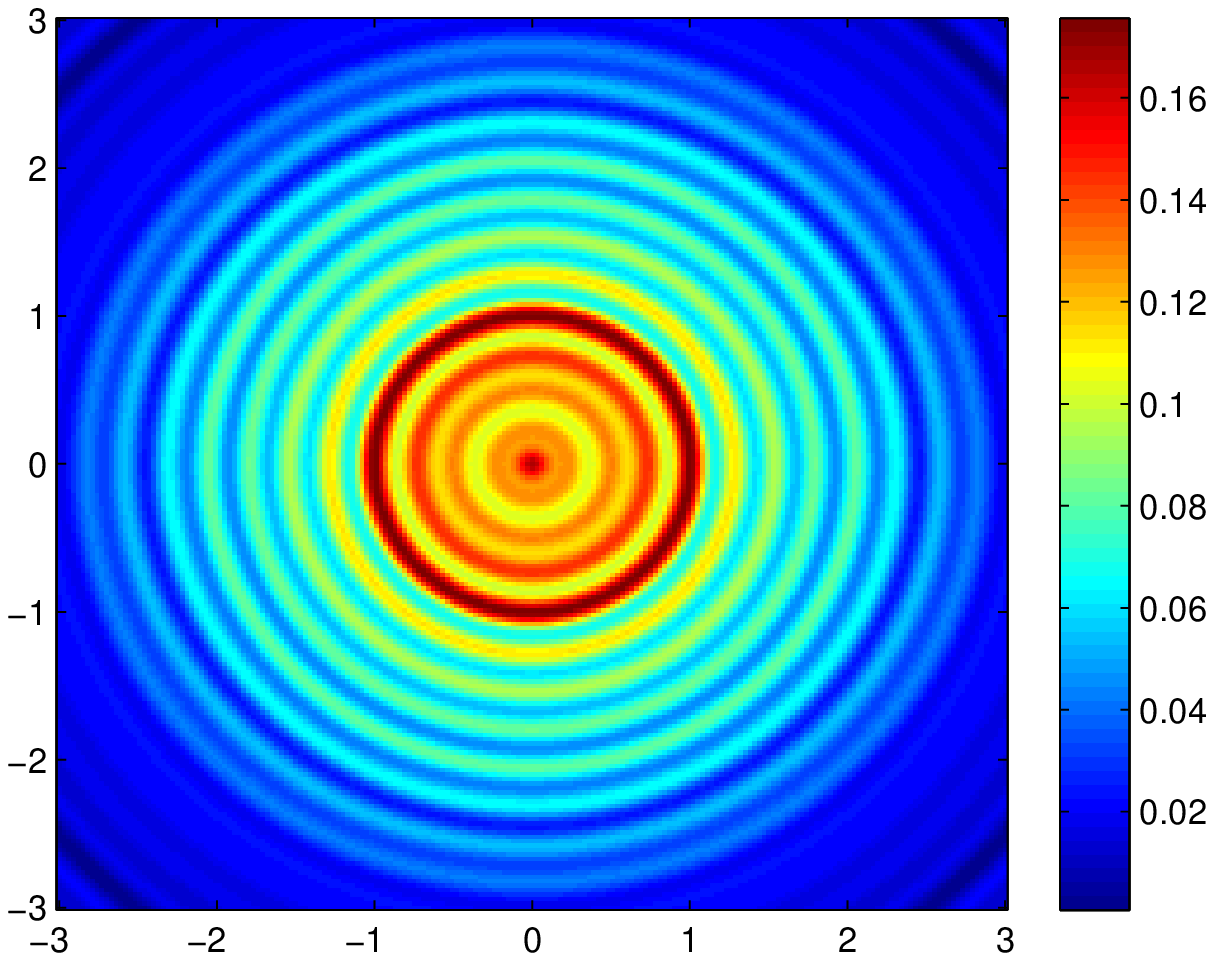}
    \includegraphics[width=0.4\textwidth]{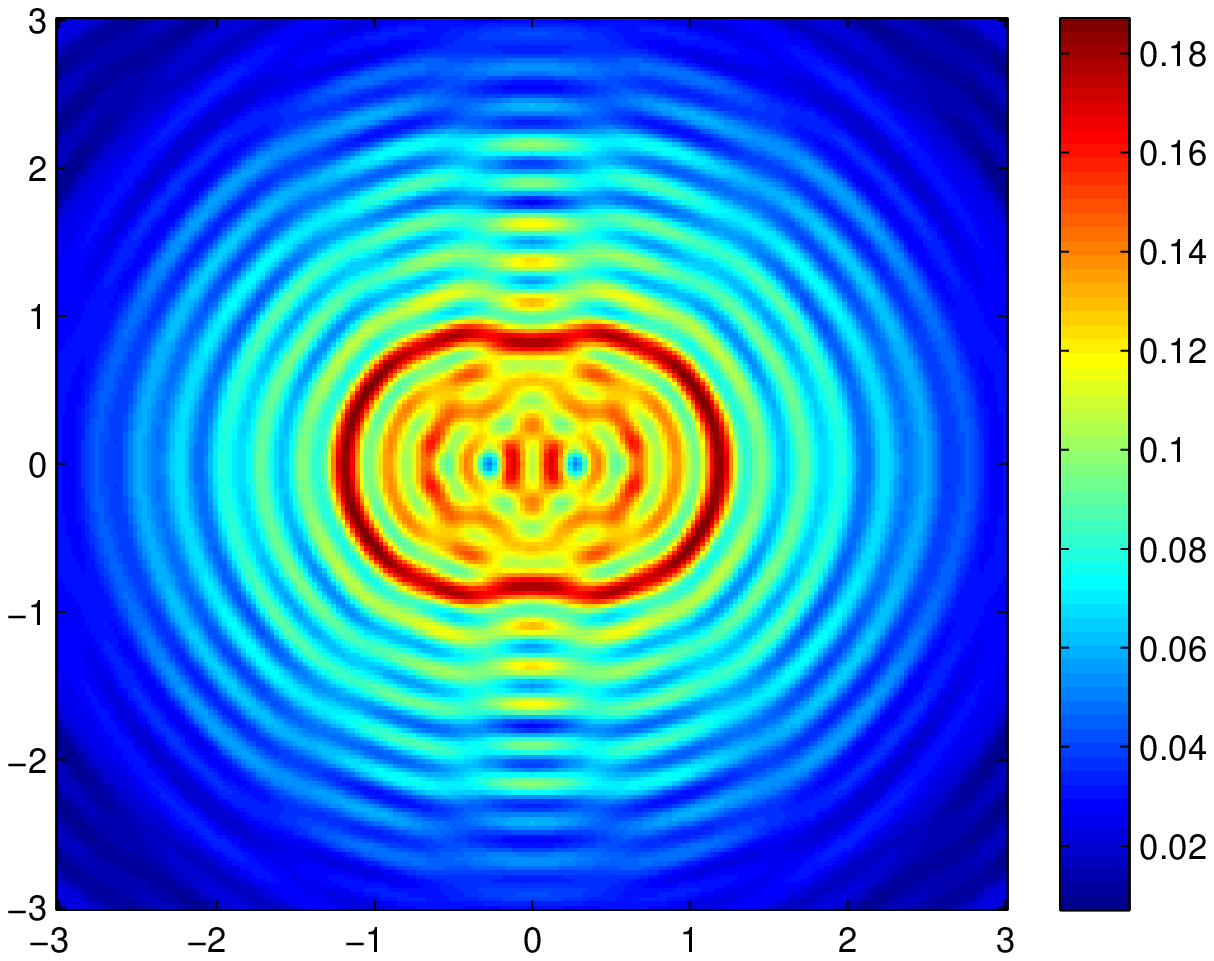} \\
        \includegraphics[width=0.4\textwidth]{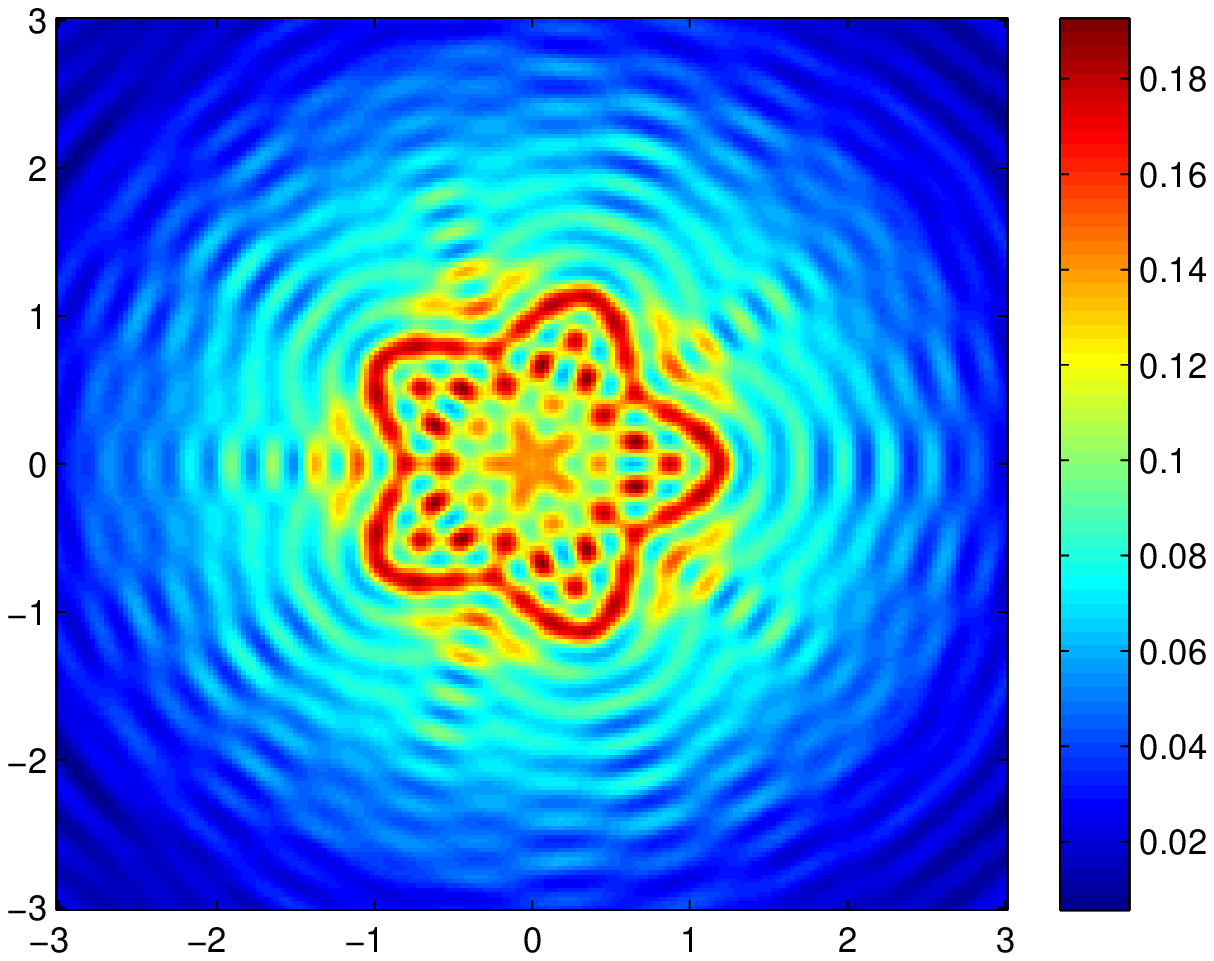}
    \includegraphics[width=0.4\textwidth]{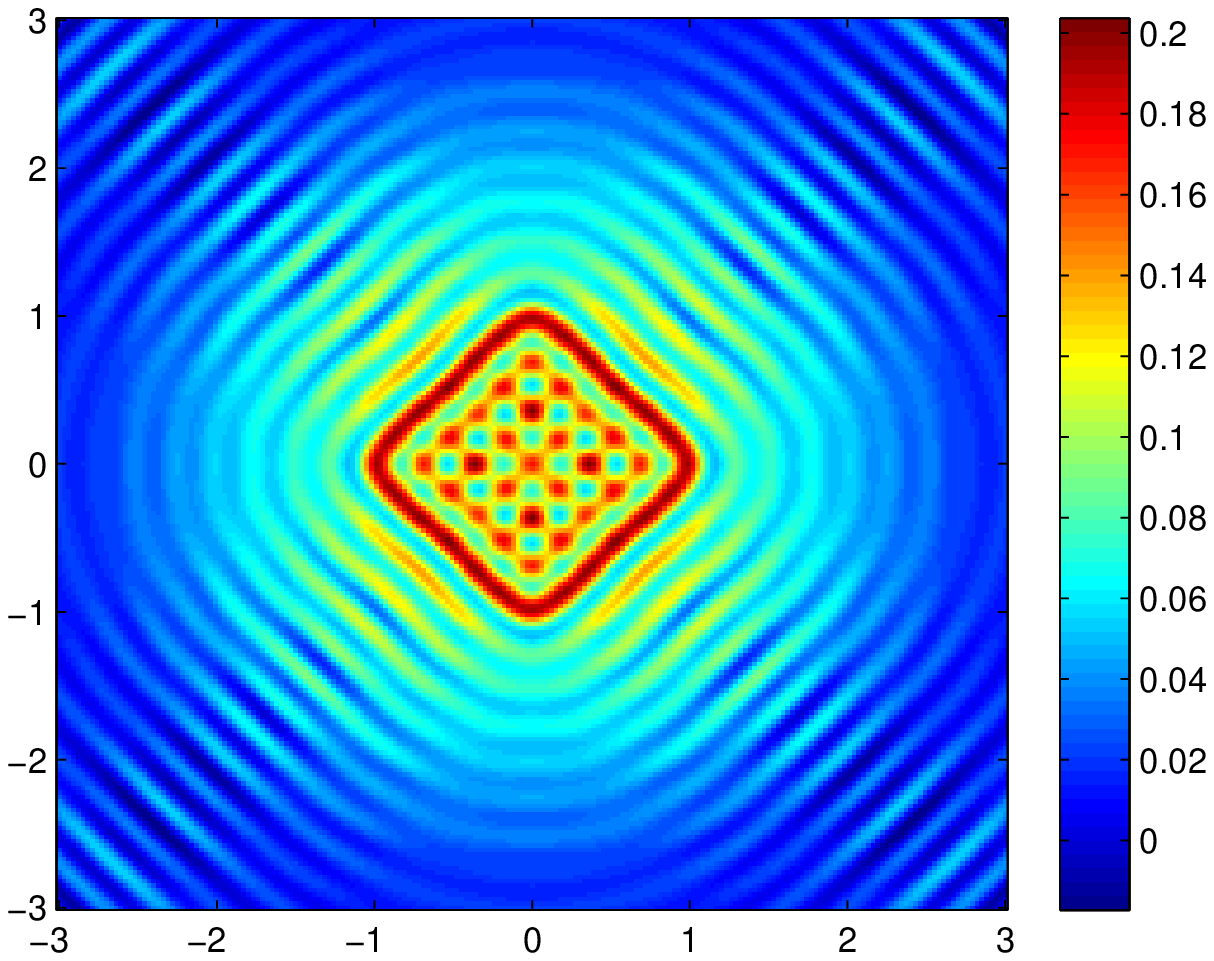}
    \caption{Example 5.1: Imaging results by RTM imaging function \eqref{scor2} with phaseless data. Top row: circle (left) and peanut (right);
     Bottom row: 5-leaf (left) and diamond (right). }\label{fig1}
\end{figure}

The imaging results are depicted in Figure \ref{fig1} which show clearly that our imaging algorithm can find
the shape and the location of the obstacles using phaseless data regardless of the shapes of the obstacles.

\begin{exmp}
We consider the imaging of a 5-leaf obstacle with impedance condition $\eta=5$, a partially coated obstacle
 with $\eta=5$ in the upper boundary and $\eta=1$ in the lower boundary,  a sound hard, and a penetrable obstacle with $n(x)=0.25$. The imaging domain is $\Om=(-3,3)\times(-3,3)$ with the sampling grid $201\times201$. The probe wave wavenumber $k=4\pi$, $N_s=N_r=128$, and $R_s=R_r=10$.
\end{exmp}

\begin{figure}
    \centering
    \includegraphics[width=0.4\textwidth]{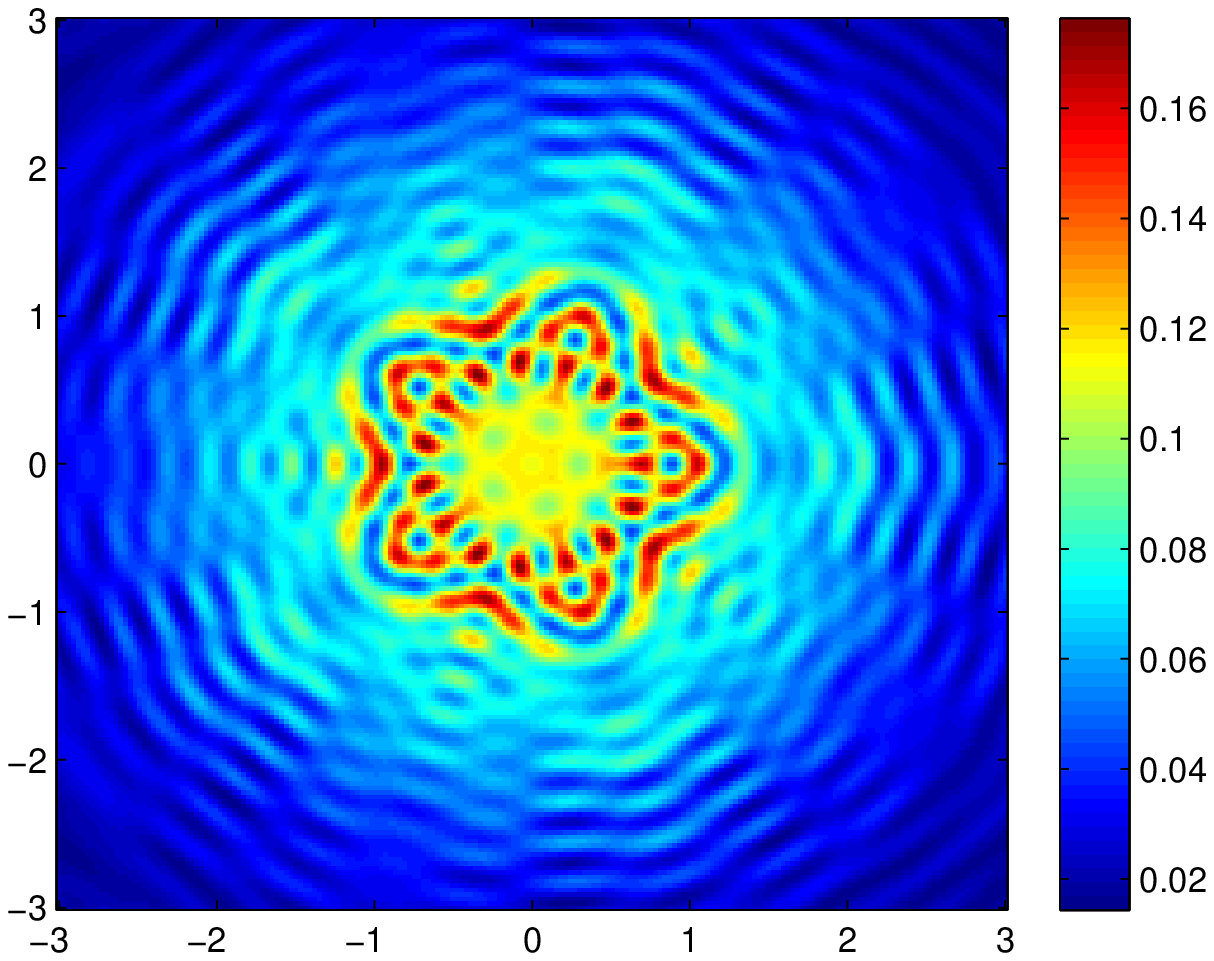}
    \includegraphics[width=0.4\textwidth]{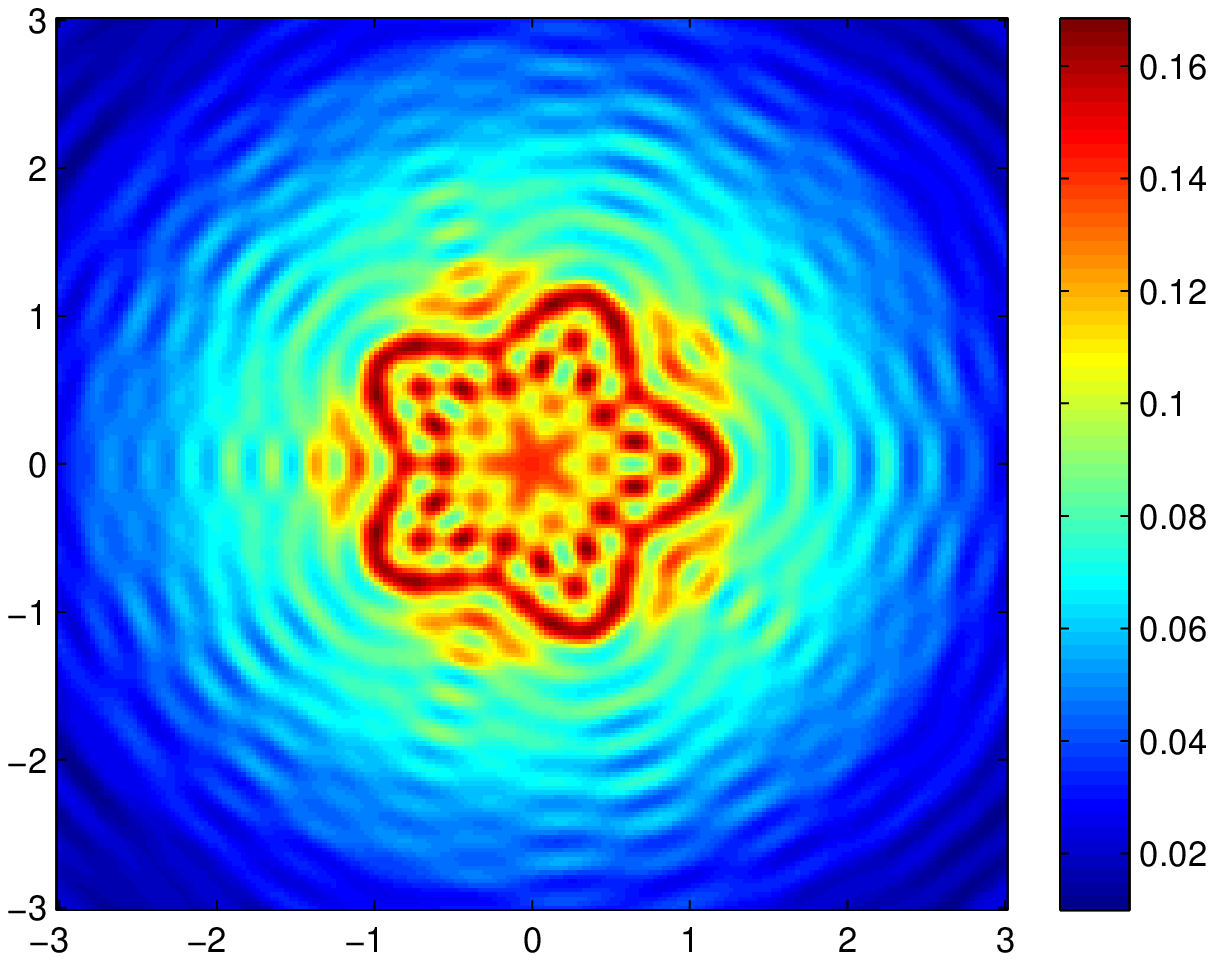}\\
   \includegraphics[width=0.4\textwidth]{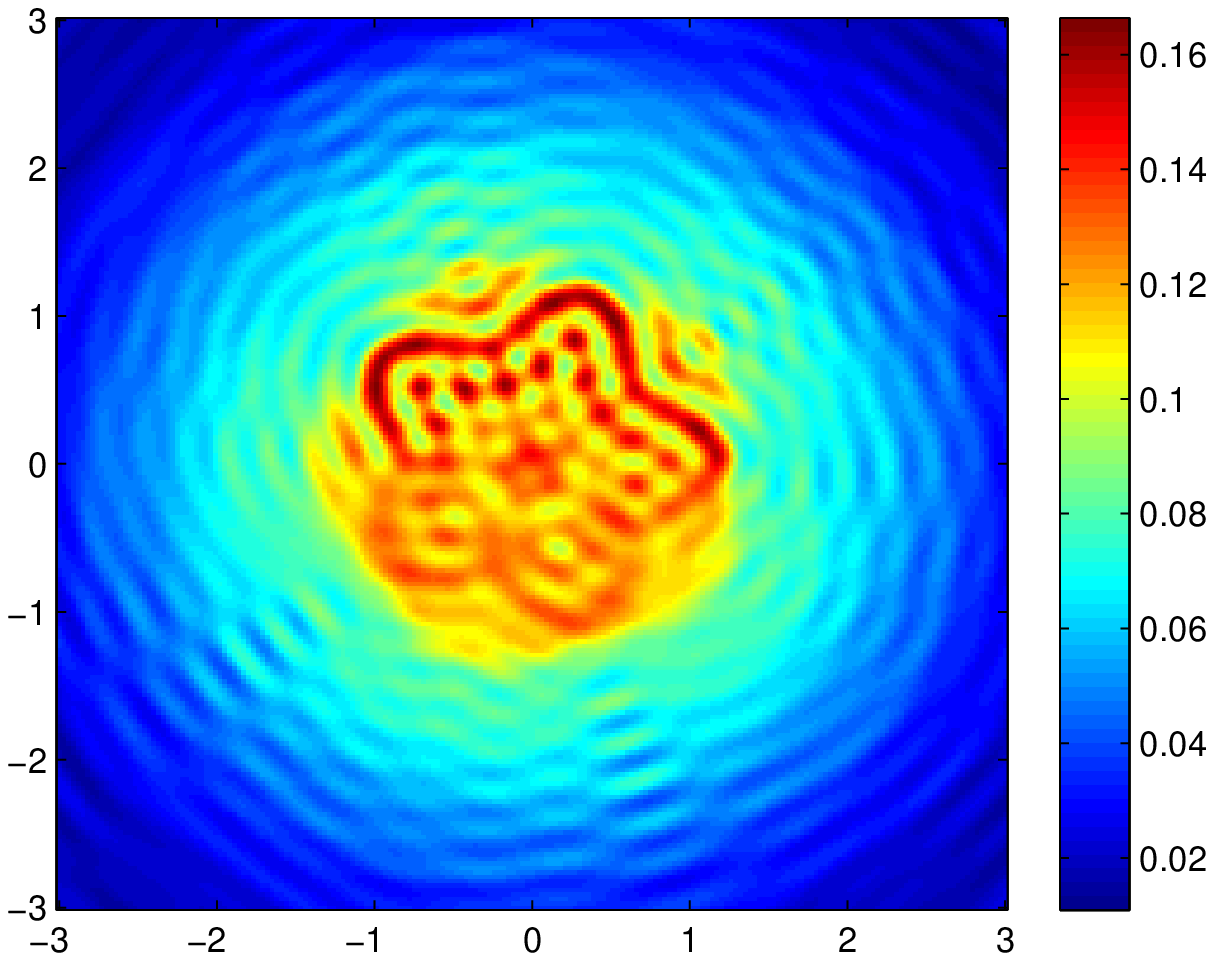}
    \includegraphics[width=0.4\textwidth]{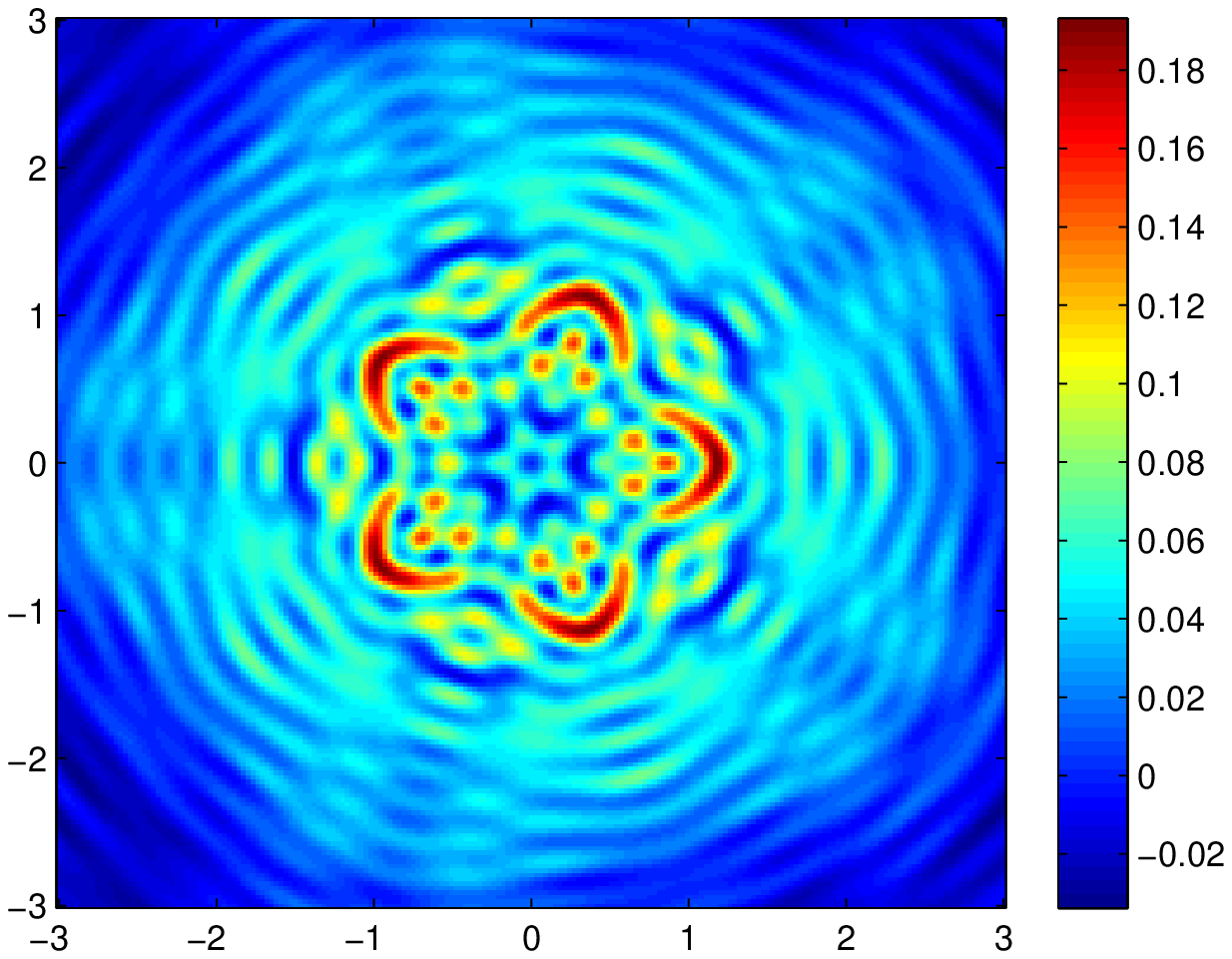}
    \caption{Example 5.2: In the top row,  a sound hard 5-leaf obstacle (left) and a non-penetrable obstacle with the impedance $\eta=5$ (right). In the bottom row, a partially coated obstacle with $\eta=5$ on the upper boundary and $\eta=1$ on the lower
    boundary (left) and a penetrable obstacle with $n(x)=1/4$ (right).} \label{fig2}
\end{figure}

Figure \ref{fig2} shows the imaging results which demonstrate clearly that our imaging algorithm
works for different types of obstacles without using any a prior information of the physical properties of the obstacles.

\begin{exmp}
We consider the stability of the imaging function with respect to the  additive Gaussian random noises
using the phaseless data. We introduce the additive Gaussian noise as follows (see e.g. \cite{cch_a}):
    \begin{equation*}
        |u|_{\rm noise} = |u| + \nu_{\rm noise},
    \end{equation*}
where $|u|$ is the synthesized phaseless total field and $\nu_{\rm noise}$ is the Gaussian noise with mean zero and standard deviation $\mu$ times the maximum of  the data $|u|$, i.e. $\nu_{\rm noise}=\mu \max|u|\eps$, and $\eps ~ \thicksim \mathcal{N}(0,1)$. \par
\end{exmp}
\begin{figure}
    \centering
    \subfigure[]{
        \includegraphics[width=0.37\textwidth]{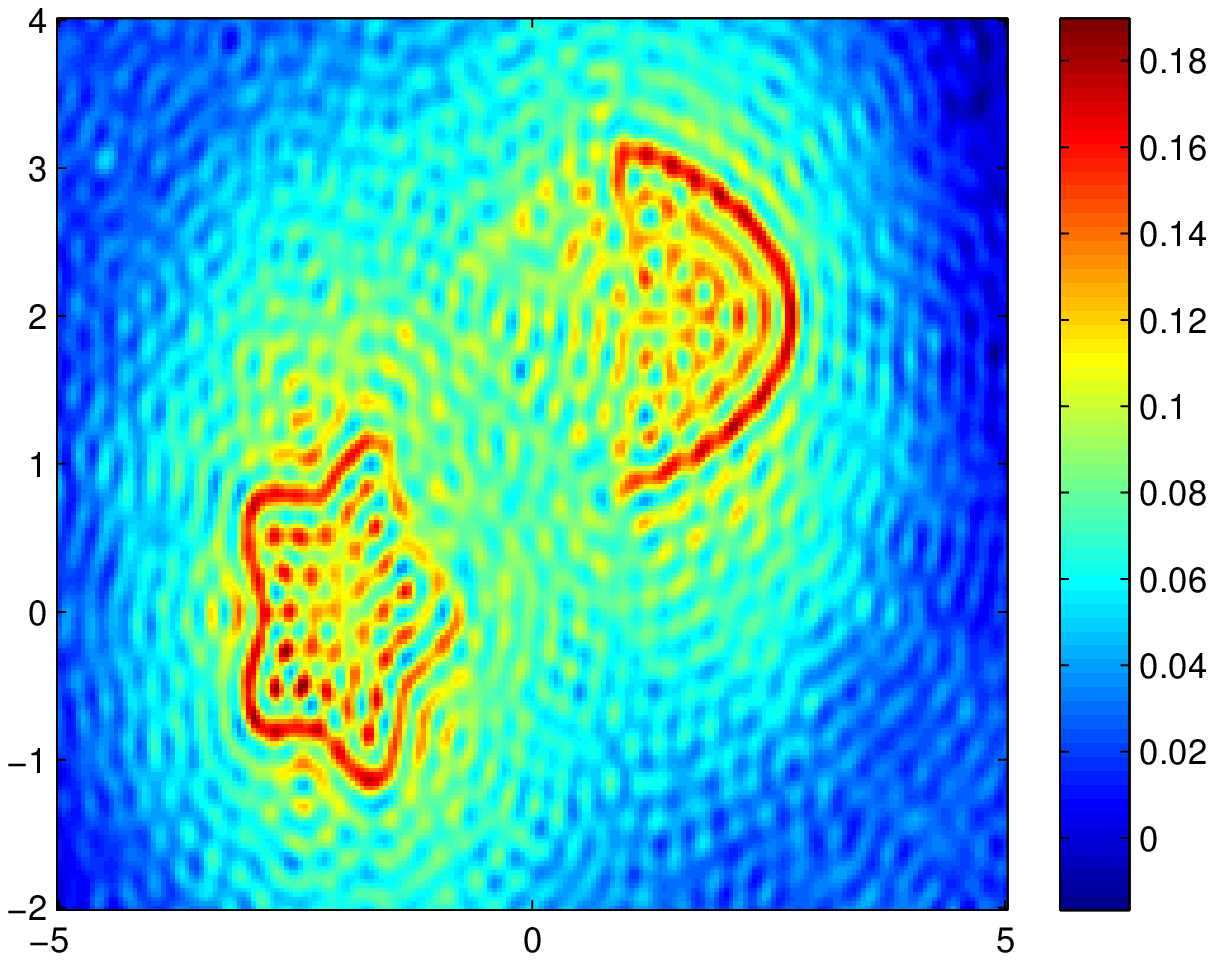}
        }
        \subfigure[]{
        \includegraphics[width=0.37\textwidth]{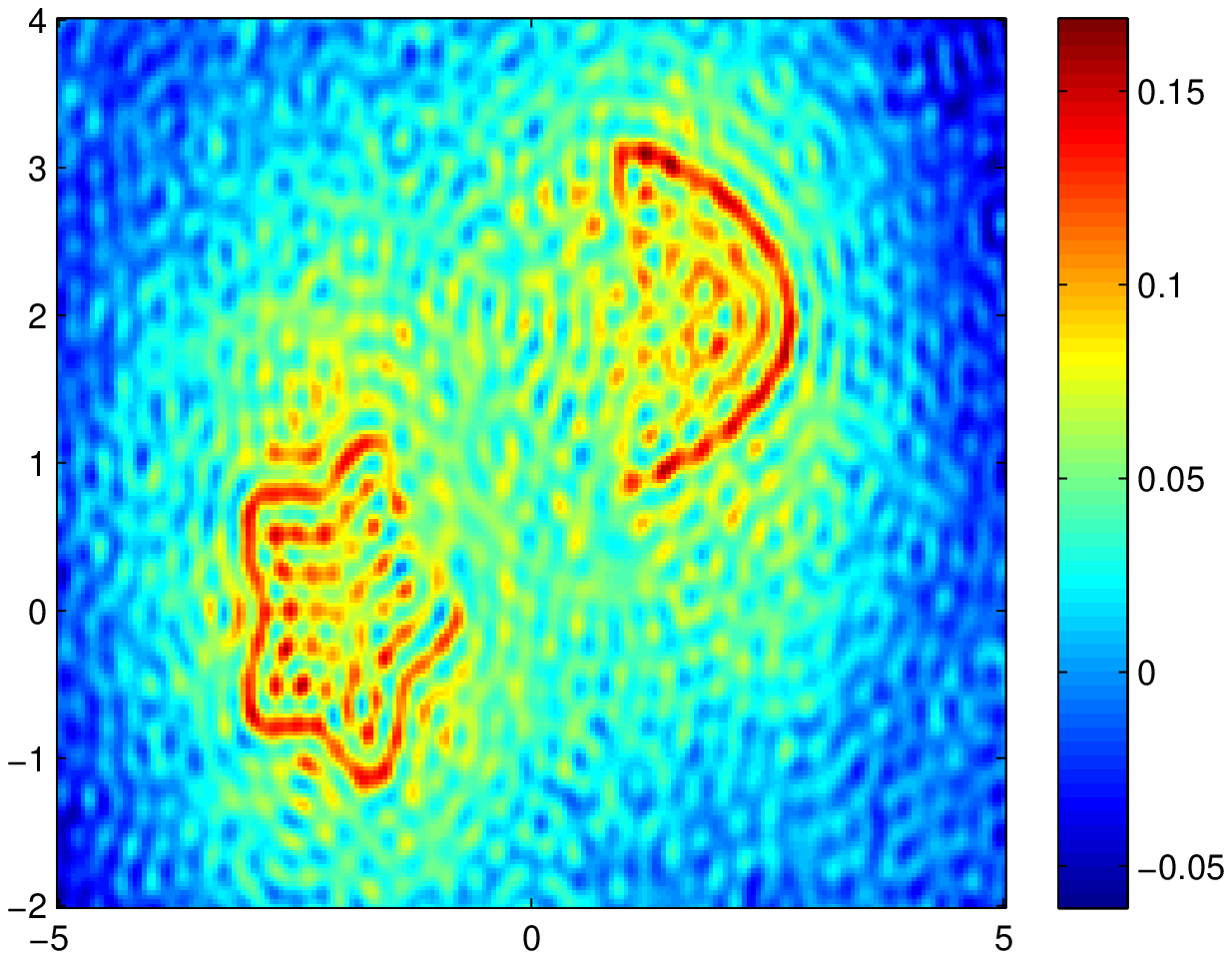}}\\
        \subfigure[]{
        \includegraphics[width=0.37\textwidth]{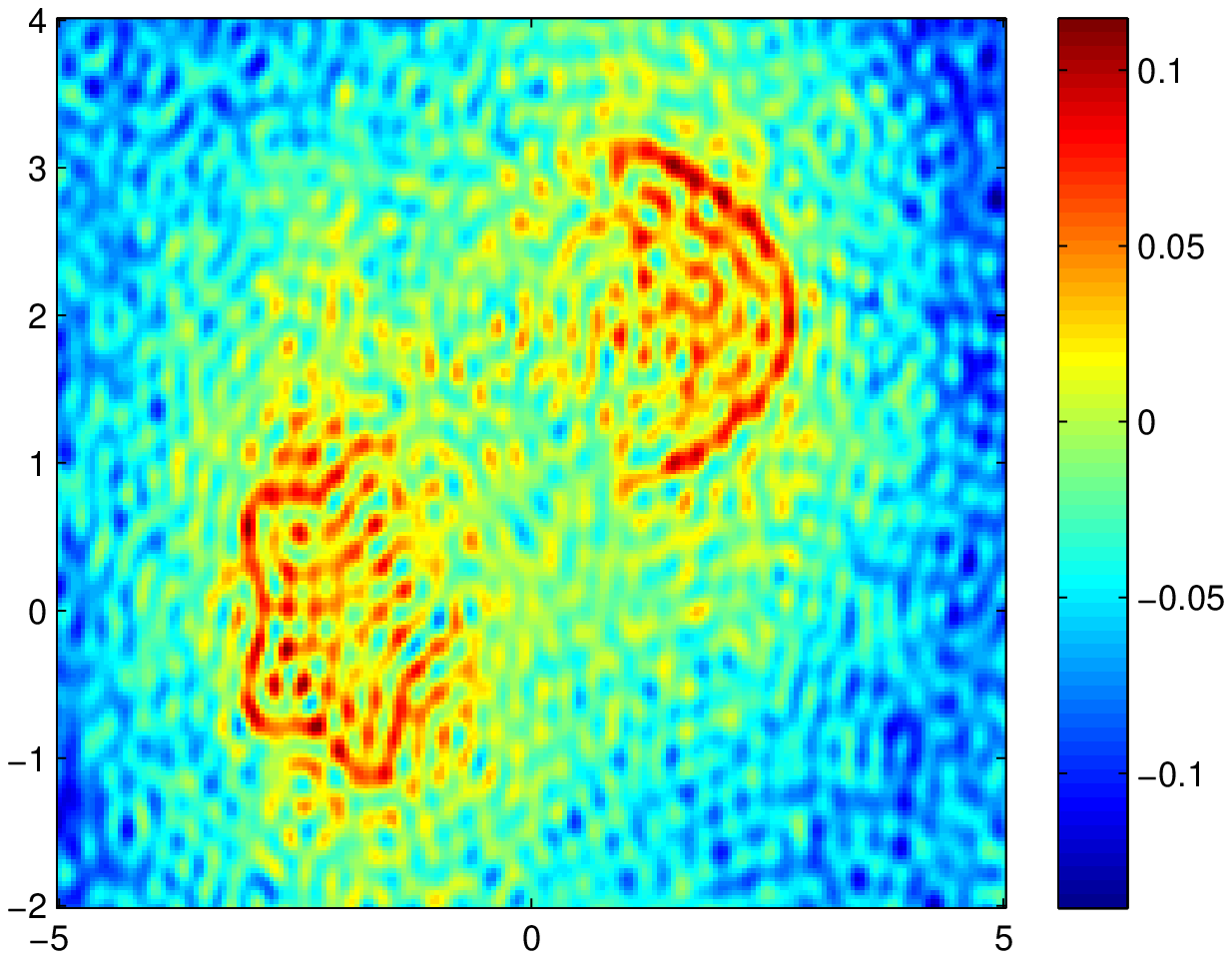}
        }
        \subfigure[]{
        \includegraphics[width=0.37\textwidth]{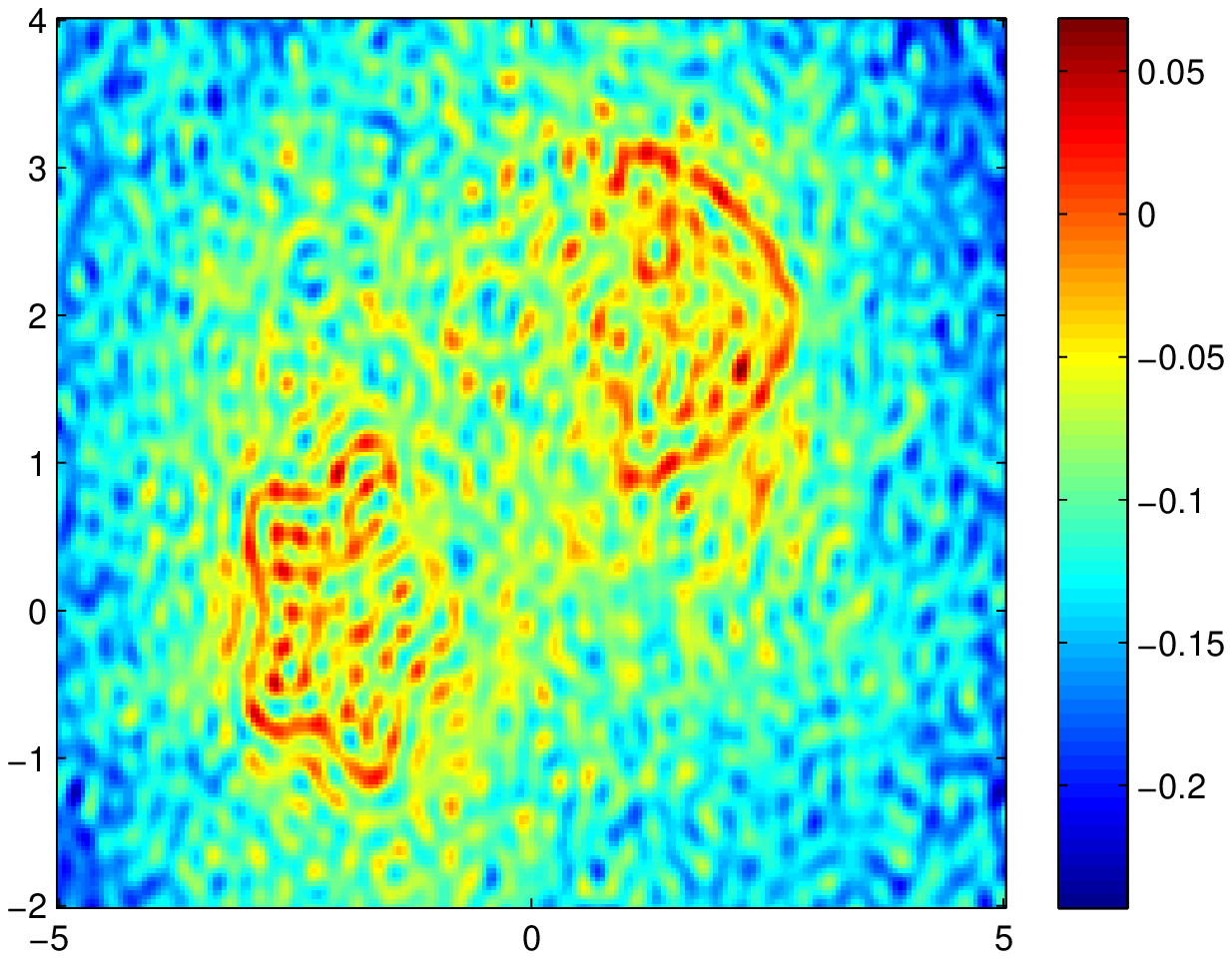}
        }
    \caption{Example 5.3: The imaging results using single frequency data added with additive Gaussian noise $\mu=10\%, 20\%, 30\%, 40\%$ from $(a)$ to $(d)$, respectively. The probe wavelength is $\lam=0.5$ and the sampling number is $N_s=N_r=256$.}\label{fig3}
\end{figure}
\begin{figure}
    \centering
    \subfigure[]{\includegraphics[width=0.37\textwidth]{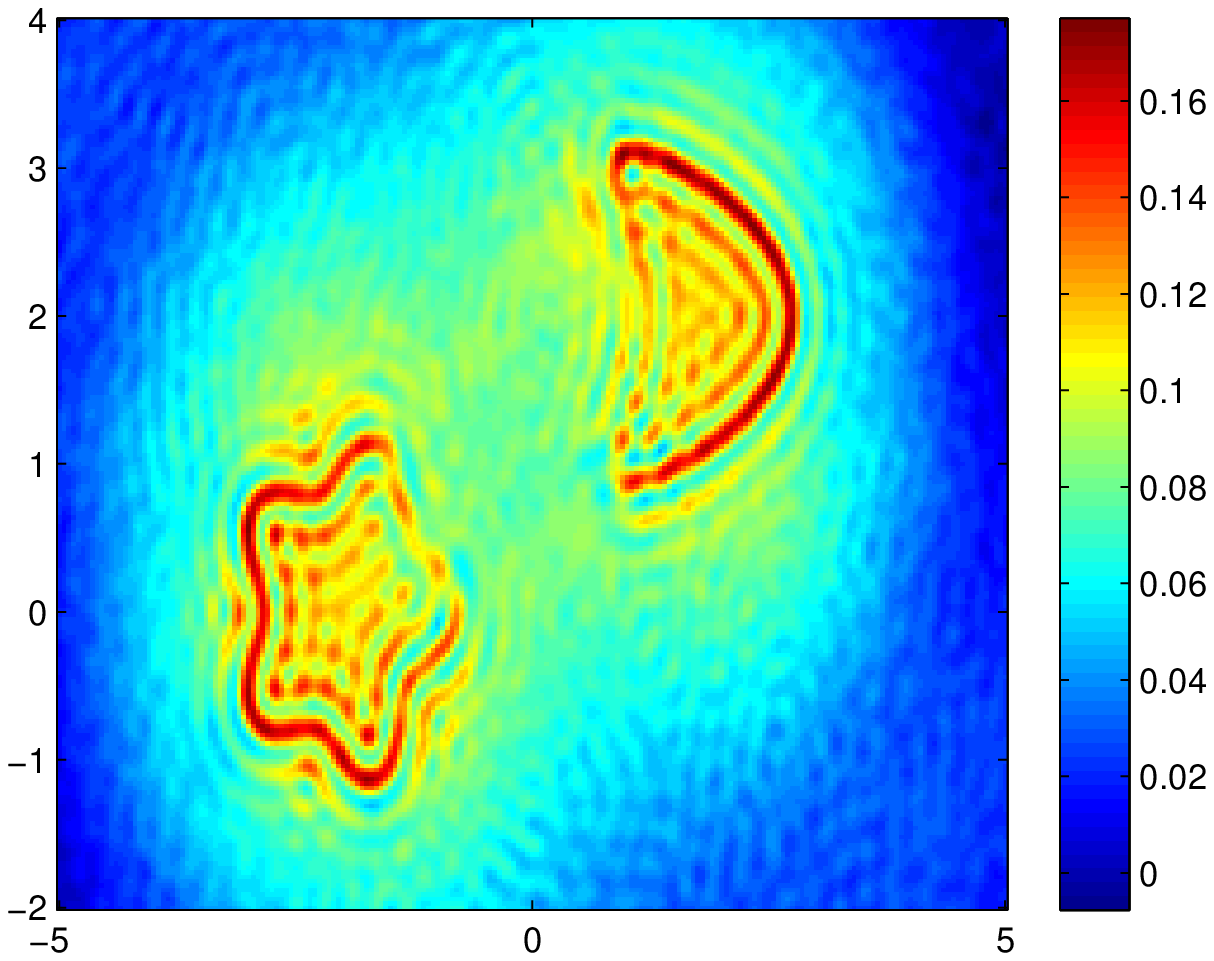}}
    \subfigure[]{\includegraphics[width=0.37\textwidth]{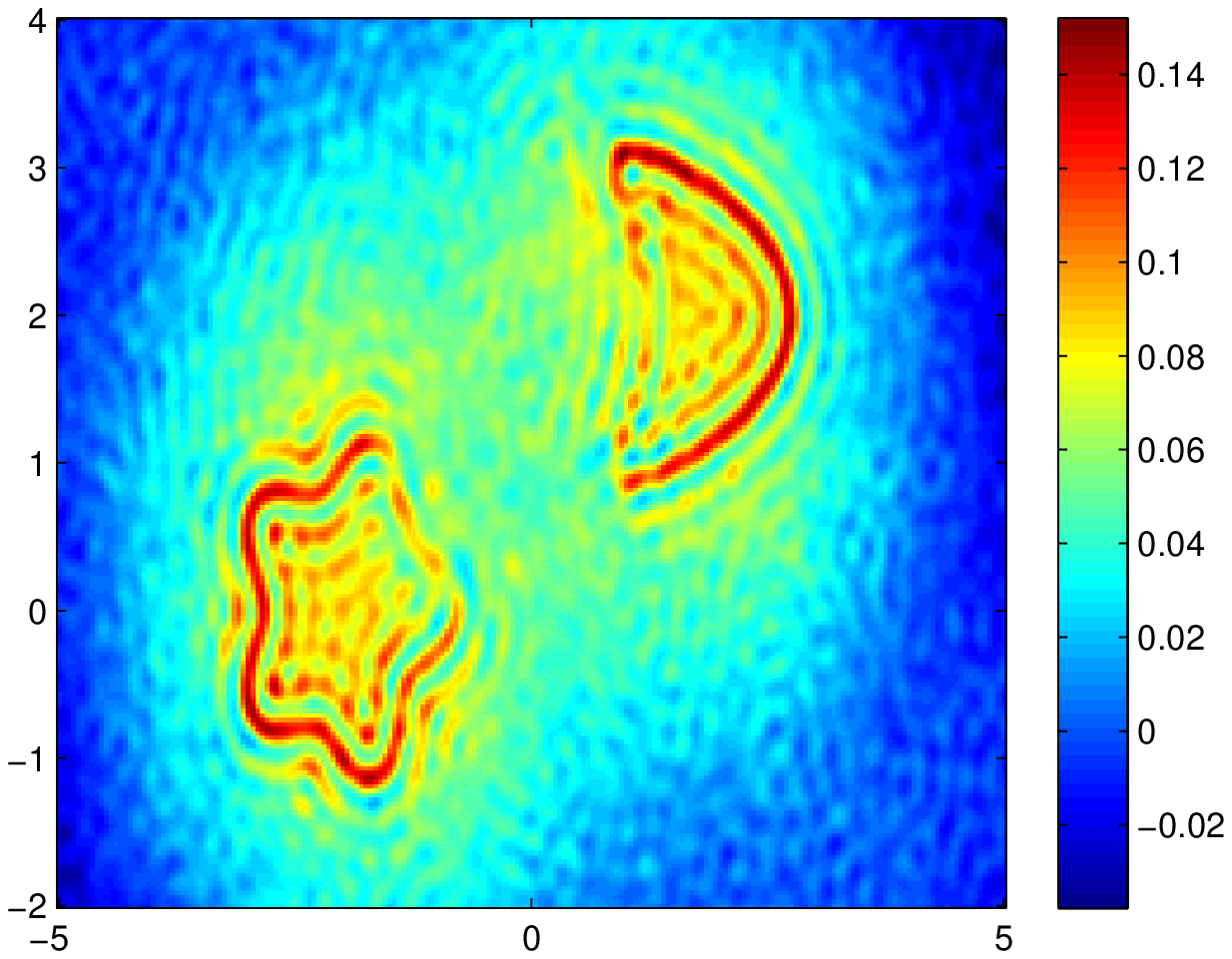}}\\
    \subfigure[]{\includegraphics[width=0.37\textwidth]{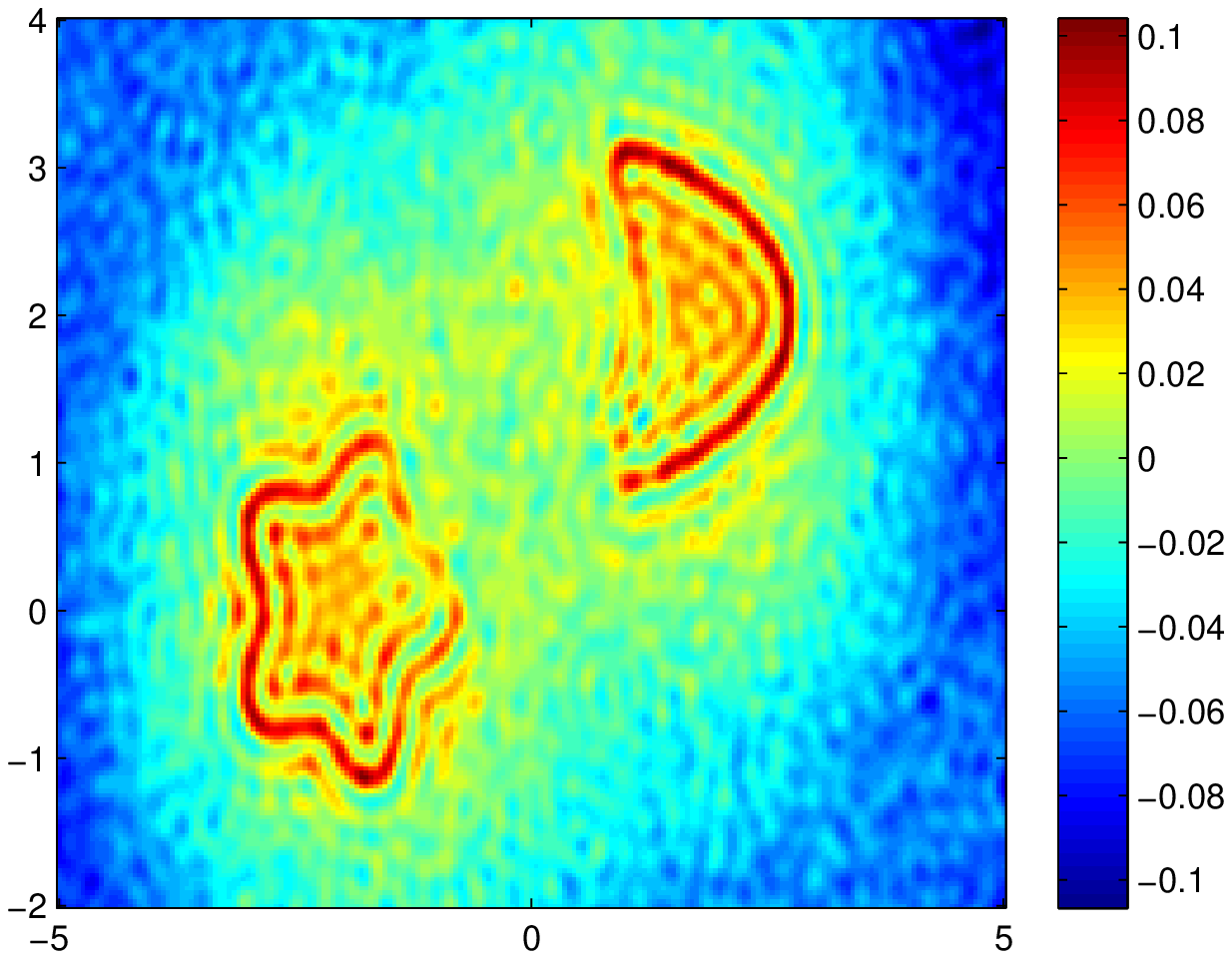}}
    \subfigure[]{\includegraphics[width=0.37\textwidth]{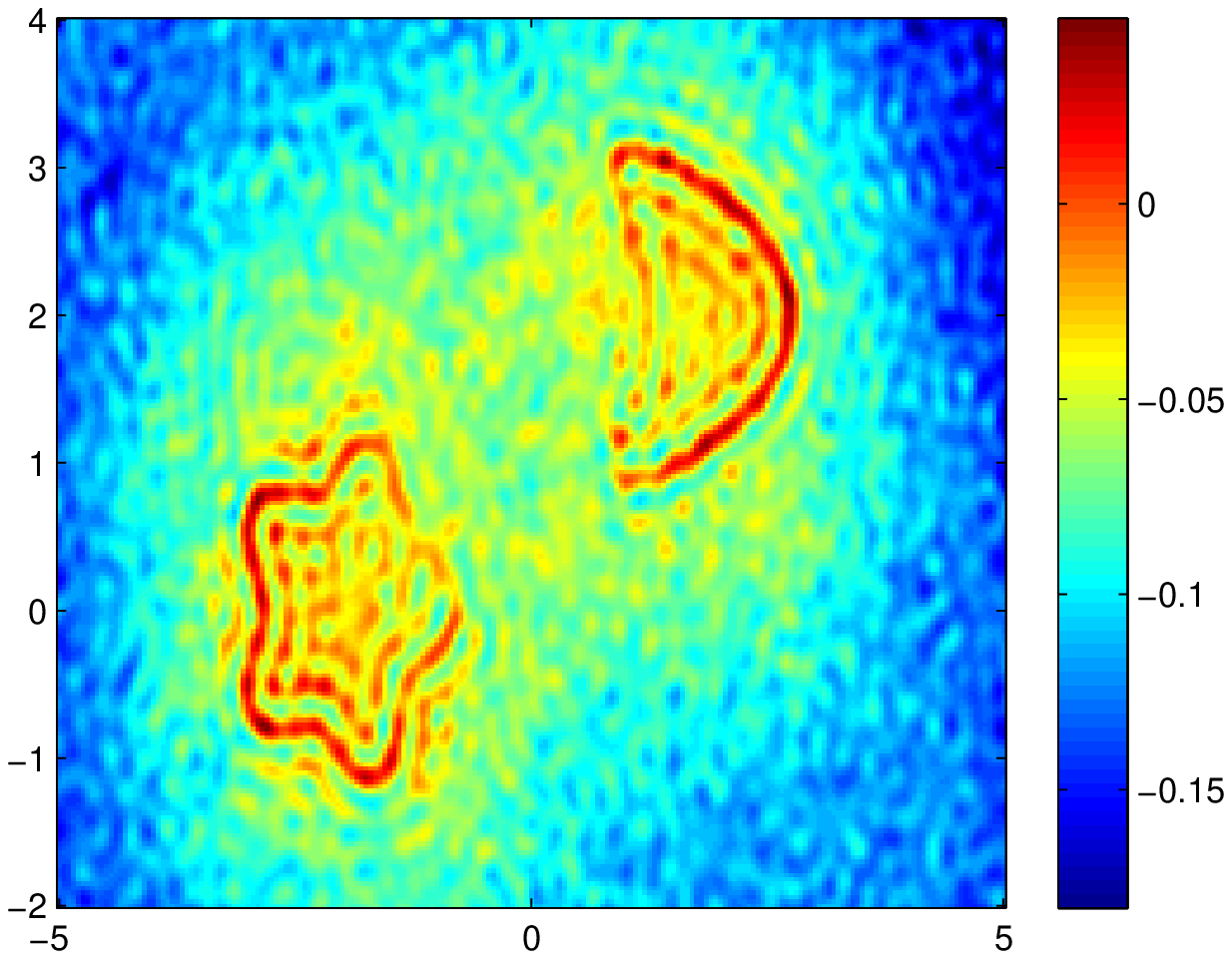}}
    \caption{Example 5.3: The imaging results using multi-frequency data added with additive Gaussian noise $\mu=10\%, 20\%, 30\%, 40\%$ from $(a)$ to $(d)$, respectively. The probe wavelengths are given by $\lam=1/1.8, 1/1.9, 1/2.0, 1/2.1,1/2.2$ and the sampling number is $N_s=N_r=256$.}\label{fig4}
\end{figure}
For the fixed probe wavenumber $k=4\pi$, we choose one kite and one 5-leaf in our test.  The search domain is $\Omega=(-5, 5)\times(-2,4)$ with a sampling $201\times 201$ mesh. We set $R_s=10, R_r=20$, and $N_s=N_r=256$.
Figure \ref{fig3} shows the imaging results for the noise level $\mu = 10\%, 20\%,30\%, 40\%$ in the single frequency data, respectively. The imaging results can be improved by superposing the multi-frequency imaging result as shown in Figure \ref{fig4}.
 The left table in Table \ref{table1} shows the noise level, where $\sigma=\mu \max_{x_r,x_s}|u(x_s,x_r)|$, $\|u\|_{\ell^2}^2=\frac{1}{N_sN_r}\sum^{N_s,N_r}_{s,r=1}|u(x_s.x_r)|^2$, $\|\nu_{\rm noise}\|_{\ell^2}^2 = \frac{1}{N_sN_r}\sum^{N_s,N_r}_{s,r=1}|\nu_{\rm noise}(x_s,x_r)|^2$.

\begin{table}[h]
\begin{center}
\begin{tabular}{ | c | c| c | c |  }
\hline
$\mu$ & $\sigma$ & $\|u\|_{\ell^2}$      & $\|\nu_{\rm noise}\|_{\ell^2}$ \\ \hline
0.1 &   0.003004    & 0.013017    & 0.003007   \\ \hline
0.2 &   0.006009    & 0.013017	  & 0.005996  \\ \hline
0.3 &   0.009013	& 0.013017    &	0.008964  \\ \hline
0.4 &   0.012018	& 0.013017    &	0.012008   \\ \hline
\end{tabular} \ \ \ \
\begin{tabular}{ | c | c| c | c |  }
\hline
$\mu$ & $\sigma$ & $\|u_s\|_{\ell^2}$      & $\|\nu_{\rm noise}\|_{\ell^2}$ \\ \hline
0.1&	0.002859    & 0.013054    & 0.002863\\ \hline
0.2&	0.005717    & 0.013054    & 0.005708\\ \hline
0.3&    0.008576    & 0.013054    & 0.008572\\ \hline
0.4&    0.011435    & 0.013054    & 0.011424\\ \hline
\end{tabular}
\end{center}
\caption{Example 5.3: The noise level in the case of single frequency data (left) and multi-frequency data (right).}\label{table1}
\end{table}

\end{document}